\documentclass[reqno,11pt]{amsart}
\pdfoutput=1
\numberwithin{equation}{section}

\usepackage[utf8]{inputenc}
\usepackage[italian,english]{babel}
\usepackage{amsfonts}
\usepackage{amssymb,esint}
\usepackage{tikz}
\usepackage{bigints}
\usepackage{comment}
\usepackage{geometry}
\usepackage{lbmath}
\usepackage[autostyle]{csquotes}

\usepackage{nccmath}
\usepackage{color}
\usepackage{mathrsfs}
\usepackage{cancel}
\usepackage{mathtools}
\usepackage{abs}
\usepackage{set}
\usepackage{norm}
\usepackage{seminorm}
\usepackage[angle-vert]{ip}
\usepackage{calligra}

\usepackage{etoolbox}
\cslet{blx@noerroretextools}\empty
\usepackage[maxbibnames=9,giveninits=true,style=alphabetic,backend=biber]{biblatex}
\DeclareFieldFormat*{title}{#1}
\usepackage[colorlinks, citecolor=citegreen, linkcolor=refred,urlcolor=blue, unicode,psdextra]{hyperref}
\usepackage{bookmark}
\usepackage{cleveref}
\expandafter\def\csname ver@etex.sty\endcsname{3000/12/31}

\usepackage{autonum}
\definecolor{mentadent}{RGB}{0,122,41}
\definecolor{citegreen}{rgb}{0,0.3,0}
\definecolor{refred}{rgb}{0.5,0,0}

\theoremstyle{plain}
\newtheorem{theorem}{Theorem}[section]
\newtheorem{lemma}[theorem]{Lemma}
\newtheorem{proposition}[theorem]{Proposition}
\newtheorem{corollary}[theorem]{Corollary}
\theoremstyle{definition}
\newtheorem{definition}[theorem]{Definition}
\theoremstyle{remark}
\newtheorem{remark}[theorem]{Remark}

\crefname{definition}{Definition}{Definitions}
\crefname{theorem}{Theorem}{Theorems}
\crefname{lemma}{Lemma}{Lemmas}
\crefname{step}{Step}{Steps}
\crefname{substep}{Step}{Steps}
\crefname{proposition}{Proposition}{Propositions}
\crefname{corollary}{Corollary}{Corollaries}
\crefname{remark}{Remark}{Remarks}
\crefname{section}{Section}{Sections}
\crefname{subsection}{Section}{Sections}
\crefname{equation}{}{}

\newcommand{\R}{\mathbb R}
\newcommand{\N}{\mathbb N}

\DeclareSymbolFont{YHlargesymbols}{OMX}{yhex}{m}{n}
\DeclareMathAccent{\arc}{\mathord}{YHlargesymbols}{"F3}

\renewcommand{\theta}{\vartheta}

\newcommand{\CC}{\mathscr{C}}
\newcommand{\Cn}{\mathrm{C}}

\newcommand{\barint}
{\rule[.036in]{.12in}{.009in}\kern-.16in \displaystyle\int\limits}

\newcommand{\Lip}{\mathrm{Lip}}

\let\div=\relax
\DeclareMathOperator{\div}{div}

\newcommand{\ee}{\mathrm{e}}
\let\oldchi=\chi
\renewcommand{\chi}{\raisebox{\depth}{\(\oldchi\)}}

\usepackage{color}

\newcommand{\numberset}{\mathbb}
\renewcommand{\N}{\numberset{N}}
\renewcommand{\R}{\numberset{R}}
\newcommand{\Sf}{\numberset{S}}

\newcommand{\B}{\mathbb{B}}

\newcommand{\dist}{\mathop {\rm dist}\nolimits}

\newcommand{\loc}{\mathrm{loc}}
\newcommand{\diam}{\mathop {\rm diam}\nolimits}

\DeclareMathOperator{\AVR}{AVR}
\let\d=\relax
\newcommand{\d}[1][g]{d_{#1}}
\newcommand{\capa}{{\rm Cap}}
\newcommand{\D}{{\rm D}}
\newcommand{\dd}{{\,\rm d}}

\newcommand{\HH}{{\rm H}}

\renewcommand{\phi}{\varphi}
\renewcommand{\epsilon}{\varepsilon}
\renewenvironment{cases}[1][l@{\ \ }l]{\arraycolsep=1.4pt\left\lbrace\kern-3pt\begin{array}{#1}}{\end{array}\right.}

\title[The asymptotic behaviour of $p$-capacitary potentials]{The asymptotic behaviour of $p$-capacitary potentials in Asymptotically Conical manifolds}

\author[L.~Benatti]{Luca Benatti}
\address{L.~Benatti, Universit\`a degli Studi di Trento,
via Sommarive 14, 38123 Povo (TN), Italy}
\email{luca.benatti@unitn.it}

\author[M.~Fogagnolo]{Mattia Fogagnolo}
\address{M.~Fogagnolo, Centro di Ricerca Matematica Ennio De Giorgi, Scuola Normale Superiore,
Piazza dei Cavalieri 3, 56126 Pisa (PI), Italy}
\email{mattia.fogagnolo@sns.it}

\author[L.~Mazzieri]{Lorenzo Mazzieri}
\address{L.~Mazzieri, Universit\`a degli Studi di Trento,
via Sommarive 14, 38123 Povo (TN), Italy}
\email{lorenzo.mazzieri@unitn.it}

\AtBeginDocument{
\hypersetup{pdftitle={The asymptotic behaviour of p-capacitary potentials in Asymptotically Conical manifolds}} 
\hypersetup{pdfauthor={Benatti L., Fogagnolo M., Mazzieri L.}} 
\hypersetup{pdfkeywords={nonlinear potential theory, inverse mean curvature flow, asymptotic behaviour, partial differential equations.}}
}
\begin{document}

\begin{abstract}
We study the asymptotic behaviour of the $p$-capacitary potential and of the weak Inverse Mean Curvature Flow of a bounded set along the ends of an Asymptotically Conical Riemannian manifolds with asymptotically nonnegative Ricci curvature. 
\end{abstract}
\maketitle


\medskip


\section{Introduction and statement of the main results}
A natural question in the qualitative study of solutions to partial differential equations regards their behaviour at large distances on complete Riemannian manifolds. For harmonic potentials, a very satisfactory description was achieved in the fairly general framework of complete Riemannian manifolds with nonnegative Ricci curvature and Euclidean Volume Growth in~\cite{Li1997} and~\cite{Colding1997}. In \cite{Li1997} using the representation formula and in \cite{Colding1997} employing the monotonicity of the Almgren's frequency function, the authors proved that the harmonic potential $u$ of an open bounded subset $\Omega$ with smooth boundary, namely the solution to 
\begin{equation}
\begin{cases}[rcl@{\ \ }l]
\Delta_g u &=&0 &\text{ on $M \smallsetminus \overline{\Omega}$,}\\
u&=&1 & \text{ on $\partial\Omega$,}\\
u(x) &\to& 0 &\text{ as $\d(x ,\Omega)\to +\infty$,}
\end{cases}
\end{equation}
is asymptotically equivalent to $\d(x, \Omega)^{2-n}$, far away from $\Omega$ (see also \cite{Ding2002,Agostiniani2018}). In \cite{Agostiniani2018} these results were applied to establish the Willmore Inequality in this framework and, consequently, a sharp Isoperimetric Inequality in dimension $n=3$ with an explicit optimal constant depending only on the Asymptotic Volume Ratio ($\AVR(g)$ for short) and the dimension of the manifold. The asymptotic behaviour of harmonic functions played a role also in the proof of the Positive Mass Theorem in \cite{Agostiniani2021} in the context of Asymptotically Flat Riemannian manifolds with nonnegative scalar curvature. 

In the last few years, it became evident that even stronger geometric conclusions can be drawn from the study of $p$-harmonic potentials on complete Riemannian manifolds, such as the validity of Minkowski Inequalities \cite{Fogagnolo2019, Agostiniani2019, Benatti2021} and the Riemannian Penrose Inequality \cite{Agostiniani2022}. Aim of the present work is to provide a detailed analysis of the asymptotic behaviour of these functions in the context of Asymptotically Conical Riemannian manifolds. This study extends some classical results, obtained  by Kichenassamy and Véron \cite{Kichenassamy1986} (see also \cite{Colesanti2015}) in the context of the flat Euclidean space.


To state our results, we now introduce some notation and setup. We recall that on a Riemannian manifold $(M,g)$ the $p$-capacitary potential of a bounded open domain $\Omega\subset M$ is the solution $u:M\smallsetminus \Omega \to \R$ to
\begin{equation}\label{pb-intro}
\begin{cases}[rcl@{\ \ }l]
\Delta^{\!(p)}_g u &=&0 &\text{ on $M \smallsetminus \overline{\Omega}$,}\\
u&=&1 & \text{ on $\partial\Omega$,}\\
u(x) &\to& 0 &\text{ as $\d(x ,\Omega)\to +\infty$,}
\end{cases}
\end{equation}
where $\Delta^{\!(p)}_g u = \div_g ( \abs{ \D u }^{p-2} \D u) $ is the $p$-Laplace operator associated with the metric $g$.
Throughout the paper we will systematically work on complete noncompact Riemannian manifolds $(M, g)$ of dimension $n \geq 3$ that are \emph{Asymptotically Conical} with \emph{quadratically asymptotically nonnegative Ricci curvature}, that is 
\begin{align}\label{eq:Ricci-intro}
    \Ric(x) \geq -\frac{(n-1)\kappa^2}{(1+\d(x,o))^2} 
\end{align}
for some some fixed $o\in M$ and $\kappa \in \R$, and every $x \in M$.
In accordance with \cite{Chodosh2017}, by the locution \emph{Asymptotically Conical Riemannian manifolds}, we are denoting,  manifolds such that, outside of some open bounded subset $K$, are diffeomorphic to a (truncated) cone $[1,+\infty) \times L$, over a smooth hypersuface $L$, called the \textit{link} of the cone, and such that the metric is close (in the $\CC^{k,\alpha}$-topology for some $k\in \N$ and $\alpha \in [0,1]$) to the cone metric $\hat{g}= \dd \rho ^2 + \rho^2 g_L$, where $\rho$ is the radial coordinate on the cone (see \cref{def:C0ACintro} for details).

Differently from the case $\Ric\geq 0$, where the Cheeger-Gromoll splitting theorem applies, here we possibly have to deal with manifolds that have more than one single end (see \cite[Definition 0.4 and discussion thereafter]{Li1992} for the notion of ends). Nevertheless, due to the compactness and smoothness of $\partial K$, the manifold is forced to have a finite number of  \emph{ends} $E_1, \ldots, E_N$. 
We can assume that  each end $E_i$ is diffeomorphic to $[1, +\infty)\times L_i$ for every $i=1, \ldots, N$, being $L_i$ the connected components of $L$, and consequently each end is Asymptotically Conical. We can define an Asymptotic Volume Ratio on each end as 
\begin{equation}
\label{eq:avrend}
    \AVR(g; E_i)= \lim_{R\to +\infty}\frac{\abs{B(o,R)\cap E_i}}{ \abs{\mathbb{B}^n}R^n}.
\end{equation}
Indeed, it is not hard to realise that, even if it is not monotone, the ratio $\abs{B(o,R)}/R^n$ has a limit as $R \to +\infty$ in Asymptotically Conical Riemannian manifolds. Moreover, the Asymptotic Volume Ratio $\AVR(g)$ of the manifold splits as
\begin{equation}
    \AVR(g)= \sum_{i=1}^N\AVR(g;E_i)
\end{equation}
where $0<\AVR(g)\leq N$ is the Asymptotic Volume Ratio of $(M,g)$.


It is not surprising that the asymptotic behaviour of the solution to \cref{pb-intro} could be different depending on the end, and in fact it turns out that the behaviour on a given end it is not affected by what happens on the others. More precisely, observing that for large enough $T$ the set $\set{u>1/T}$ contains $K$, we can define the \emph{normalised $p$-capacity of $\Omega$ relative to the end $E_i$} as
\begin{equation}
\label{eq:pcapi}
    \Cn^{(i)}_p(\Omega) =\left(\frac{p-1}{n-p}\right)^{p-1} \frac{1}{\abs{\S^{n-1}}} \int\limits_{\set{u=1/t}\cap E_i} \abs{\D u}^{p-1} \dd\sigma \, ,
\end{equation}
for every $i=1,\ldots,m$. Observe that, by virtue of the divergence theorem, the right hand side does not depend on $t>T$ and the above formula yields a well posed definition. The \emph{normalised $p$-capacity of $\Omega$} (see \cref{def:capandnormpcap}) then split as
\begin{equation}
    \Cn_p(\Omega)= \sum_{i=1}^N \Cn^{(i)}_p(\Omega).
\end{equation}
We are now ready to state our first main result.

\begin{theorem}[Asymptotic behaviour of the $p$-capacitary potential]\label{thm:asymptoticbehaviourofppotentialC0}
Let $(M,g)$ be a complete $\CC^0$-Asymptotically Conical Riemannian manifold with Ricci curvature satisfying \cref{eq:Ricci-intro}. Let $E_1, \ldots, E_N$ be the (finitely many) ends of $M$ with respect to the compact $K$ in \cref{def:C0ACintro}. Consider $\Omega\subset M$ be an open bounded subset with smooth boundary and $u:M\smallsetminus \Omega \to \R$ the solution to the problem \eqref{pb-intro}. Then
\begin{align}\label{eq:p-cap_asymptotic_behaviour_C0}
    u(x) &= \left( \frac{ \Cn^{(i)}_p( \Omega)}{ \AVR(g; E_i)}\right)^{\frac{1}{p-1}} \rho(x)^{-\frac{n-p}{p-1}}+o\left(\rho(x)^{-\frac{n-p}{p-1}} \right)
\end{align}
on $E_i$ as $\rho(x)\to +\infty$ for every $i=1, \ldots, N$.
\end{theorem}

When the $\CC^{0}$-Asymptotically Conical condition is strengthen to $\CC^{k}$-Asymptotically Conical condition, it is not difficult to deduce corresponding asymptotic behaviours for the derivatives up to order $k$ of $u$ (see \cref{thm:asymptoticbehaviourofpotential} below). Our result extends to the nonlinear setting the asymptotic analyses carried out in \cite[Theorem 2.2]{Agostiniani2020}, \cite[Lemma 2.2]{Agostiniani2021}, \cite[Lemma 4.1]{Hirsch2020} and \cite[Lemma A.2]{Mantoulidis2020}, although without refined estimates of the error terms. It would be interesting to deal, building on \cref{thm:asymptoticbehaviourofppotentialC0}, with such remainders, possibly following the lines of \cite{Chrusciel1990}.

We emphasise that \cref{thm:asymptoticbehaviourofppotentialC0} plays a crucial role in the recent proof of the Riemannian Penrose Inequality through $p$-harmonic potentials proposed in \cite{Agostiniani2022} and that it will be employed to provide new results in this field under milder asymptotic conditions. To this end, we point out that the requirements above do not involve explicit rates of decay to the reference metric, that are usually assumed when dealing with these topics.



\medskip

It turns out that our approach, employed for proving \cref{thm:asymptoticbehaviourofppotentialC0} and its consequences, appearing in \cref{sec:asy-p}, happen to fit also the geometric case of the weak Inverse Mean Curvature Flow starting at a bounded $\Omega \subset M$ with smooth boundary.
We briefly recall that with this notion, introduced by Huisken and Ilmanen \cite{Huisken2001} as a weak counterpart to the classical evolution by inverse mean curvature \cite{Gerhardt1990, Urbas1990}, it is indicated a Lipschitz function $w \in \mathrm{Lip}_{\mathrm{loc}}$
that satisfies 
\[
\mathrm{div}\left(\frac{\mathrm{D} w}{\abs{\mathrm{D} w}}\right) = \abs{\mathrm{D} w}
\]
on $M \setminus \overline{\Omega}$ and such that $\Omega = \{w < 0\}$ in a very geometric nonstandard weak variational sense. By the pioneering work of Moser \cite{Moser2007,Moser2008} and subsequent extensions to Riemannian manifolds \cite{Kotschwar2009, Mari2019}, the solution $w$ can also be interpreted as the locally uniform limit as $p \to 1^+$ of $-(p-1) \log u_p$, where $u_p$ is the $p$-capacitary potential of $\Omega$. 

In analogy with \cref{eq:pcapi}, we set
\begin{equation}
     \abs{\partial \Omega^*}^{(i)}= \frac{\abs{\partial\set{w\leq t}\cap E_i}}{\ee^t},
\end{equation}
for every $i=1, \ldots, N$ and every $t \geq T$, where $T$ is so chosen that $\set{w \geq T}$ contains $K$.
By means of \cite[Exponential Growth Lemma 2.6]{Huisken2001} and the divergence theorem, we have that the right hand side does not depend on $t >T$ and yields a well posed definition. As for the $p$-capacity, we have that 
\begin{equation}
\abs{\partial \Omega^*}= \sum_{i=1}^N \abs{\partial \Omega^*}^{(i)}
\end{equation}
where $\Omega^*$ is the strictly outward minimising hull of $\Omega$ \cite{Fogagnolo2020a}. We refer to \cref{sec:asy-imcf} for a  more detailed discussion.
The following theorem provides a description of the asymptotic behaviour for the weak IMCF for large times.

\begin{theorem}[Asymptotic behaviour of the Inverse Mean Curvature Flow]\label{thm:asymptotic_behaviour_of_IMCF}
Let $(M,g)$ be a complete $\CC^1$-Asymptotically Conical Riemannian manifold with Ricci curvature satisfying \cref{eq:Ricci-intro}. Let $E_1, \ldots, E_N$ be the (finitely many) ends of $M$ with respect to the compact $K$ in \cref{def:C0ACintro}. Consider $\Omega\subset M$ be an open bounded subset with smooth boundary and $w:M\smallsetminus \Omega \to \R$ be the weak Inverse Mean Curvature Flow  starting at $\Omega$. Then
\begin{align}\label{eq:asymptotic_behaviour_IMCF}
    w(x) &= \log\left(\rho(x)^{n-1}\right)- \log\left(\frac{ \abs{\partial \Omega^*}^{(i)}}{\AVR(g;E_i) \abs{ \S^{n-1}}}\right) +o\left(1 \right)
\end{align}
on $E_i$ as $\rho(x)\to +\infty$ for every $i=1, \ldots, N$.
\end{theorem}
It might be useful to observe that, coherently, multiplying by $-(p-1)$ the logarithm of the right hand side of \cref{eq:p-cap_asymptotic_behaviour_C0} and letting $p \to 1^+$ one exactly recovers the right hand side of \cref{eq:asymptotic_behaviour_IMCF}. In fact, as a consequence of \cite[Theorem 1.2]{Fogagnolo2020a}, we have that
\begin{equation}
    \label{eq:convergence_p-cap_potential_SOMH}
    \lim_{p \to 1^+} \mathrm{C}^{(i)}_p (\Omega) = \frac{\abs{\partial \Omega^*}^{(i)}}{\abs{\Sf^{n-1}}}
\end{equation}
holds for every $i=1, \ldots, N$ (see \cref{lem:convergenceipcap} below). Observe that, differently from \cref{thm:asymptoticbehaviourofppotentialC0} we required $\CC^1$-convergence of the metric. This additional requirement is due to the fact that, up to the authors' knowledge, a Cheng-Yau-type estimate with sharp decay for the gradient of the IMCF is not known. Therefore, in our proof we use the gradient bound \cite[Weak Existence Theorem 3.1]{Huisken2001}. In some cases, for example $\Ric \geq 0$, this requirement can be weakened in favour of the $\CC^0$-convergence (see the discussion before \cref{prop:Cheng-Yau-Improved}). To our knowledge,  \cref{thm:asymptotic_behaviour_of_IMCF} with the explicit constant in the expansion  \eqref{eq:asymptotic_behaviour_IMCF} was known only in the flat case of $\R^n$, and was obtained by completely different means. Indeed, in this setting, the level sets of the weak IMCF become starshaped (and thus smooth) after a sufficiently long time, as a consequence of \cite[Theorem 2.7]{Huisken2008}. At this point, \cref{eq:asymptotic_behaviour_IMCF} could be easily deduced by classical results \cite{Gerhardt1990, Urbas1990} for the smooth IMCF. It is worth pointing out that the arguments we employ got an important inspiration also from those in the proof of \cite[Blowdown Lemma 7.1]{Huisken2001}, that actually helped also in establishing \cref{thm:asymptoticbehaviourofppotentialC0}. \cref{thm:asymptotic_behaviour_of_IMCF} provided a simplified approach to the proof of \cite[Blowdown Lemma 7.1]{Huisken2001}. At the same time extend such result to the class of Asymptotically Conical Riemannian manifolds, adding the explicit characterisation of the constant term in the expansion \eqref{eq:asymptotic_behaviour_IMCF}.

\smallskip
The results above in particular apply to Asymptotically Locally Euclidean spaces (ALE for short) {\em gravitational instantons}, that are noncompact hyperkh\"aler Ricci Flat $4$-dimensional manifolds playing a role in the study of Euclidean Quantum Gravity Theory, Gauge Theory and String Theory (see~\cite{Hawking1977,Eguchi1979,Kronheimer1989,Kronheimer1989a,Minerbe2009,Minerbe2010,Minerbe2011}). Moreover, it is not difficult to realise that the completeness assumption can be dropped, and that the above results can be extended to manifolds with boundary. Indeed, the method proposed is completely blind to everything is inside $\Omega$, and, therefore, we can also include Asymptotically Flat Riemannian manifolds $(M, g)$ with compact boundary $\partial M$, of fundamental importance in General Relativity.
\smallskip

\bigskip

\noindent\textbf{Summary.} The paper is organised as follows. In \cref{main-sec}, we recall some basic notions about $p$-harmonic functions and the $p$-capacitary potential as well as an improvement of Li-Yau-type estimates holding true on Asymptotically Conical Riemannian manifolds satisfying the bound \cref{eq:Ricci-intro}, with controlled constants as $p\to 1^+$. \cref{sec:asy-p} and \cref{sec:asy-imcf} are devoted to the proof of the asymptotic behaviour of the $p$-capacitary potential  and of the (weak) IMCF, that are respectively \cref{thm:asymptoticbehaviourofppotentialC0} and \cref{thm:asymptotic_behaviour_of_IMCF}, together with some other related results. In the last section we prove the sharpness of the Minkowski Inequality in Asymptotically Conical Riemannian manifolds with nonnegative Ricci curvature and its rigidity statement in the general Euclidean volume growth case.

\medskip


\noindent\textbf{Acknowledgements.} 
\emph{The authors are grateful to V. Agostiniani, F. Oronzio, G. Antonelli for precious discussions and comments during the preparation of this manuscript. The authors are members of Gruppo Nazionale per l'Analisi Matematica, la Probabilit\`a e le loro Applicazioni (GNAMPA).}

\section{Preliminary results on \texorpdfstring{$p$}{p}-capacitary potentials}
\label{main-sec}

\subsection{\texorpdfstring{$p$}{p}-harmonic functions and regularity} \label{reg-subsec}
Before considering the specific case of problem \eqref{pb-intro}, we recall the definition of $p$-harmonic functions, as well as their regularity estimates. Given an open subset $U$ of a complete Riemannian manifold $(M, g)$, we say that $v \in W^{1, p}(U)$ is $p$-harmonic if 
\begin{equation}
\label{def-p-harm}
\int\limits_U \left\langle\abs{\D v}^{p-2} \D v \, \vert \, \D \psi\right\rangle \dd \mu =0.
\end{equation}
for any test function $\psi \in \CC_c^\infty(U)$. With $\langle\, \cdot \, \vert \, \cdot\, \rangle$ we denote as usual the scalar product induced by the underlying Riemannian metric $g$ on the tangent space at each point.
  
We now precisely recall what is known about the regularity of $p$-harmonic functions. 
Suppose by now that $K\Subset U$ is entirely contained in a chart of $M$. For any $k \in \N$ and $\alpha \in (0,1)$, we have that
\begin{align}\label{eq:uniformellipticity}
    g_{ij}(x)\xi^i\xi^j>\lambda \delta_{ij}\xi^i\xi^j && \norm{g_{ij}}_{\CC^{k,\alpha}(K)}<\Lambda_{k,\alpha}
\end{align}
holds for every $x\in K$ and $\xi \in \R^n$ for some positive constants $\lambda$ and $\Lambda_{k,\alpha}$. Regularity results for $p$-harmonic functions (see \cite{Tolksdorf1983,DiBenedetto1983,Lieberman1988}) yield
\begin{align}
\abs{\D v(x)} \leq \Cn  &&\abs{\D v(x)-\D v(y)} \leq \Cn \,\left(\vphantom{\sum} \d(x,y)\right)^\beta
\end{align}
for every $p$-harmonic function $v$ with $\abs{v} \leq 1$, with $\Cn$, $\beta$ depending only on $n$, $p$, $\alpha$, $\lambda$ and $\Lambda_{0,\alpha}$ and the distance of $K$ from the boundary of $U$. For a general $K$, since a compact set can be finitely covered by charts, the result can be extended defining, with abuse of notation, $\lambda$ and $\Lambda_{k,\alpha}$ as the minimal $\lambda$ and maximal $\Lambda_{k,\alpha}$ in \eqref{eq:uniformellipticity} among a family of charts covering $K$. Hence, the theorem below easily follows by a scaling argument, being that $v/\norm{v}_{L^\infty(U)}$ is $p$-harmonic as well.

\begin{theorem}[Schauder interior estimates]\label{thm:regularityp-potential}
Let $(M,g)$ be a complete Riemannian manifold of dimension $n$, $U\subseteq M$ be an open subset and let $1<p$. For any $\alpha\in(0,1)$ and $K\Subset U$, there exists a positive constant $\beta=\beta(n,p,\alpha, \Lambda_{0,\alpha}, \lambda)\in (0,\alpha)$ such that any bounded solution $v:U \to \R$ of the problem $\Delta_g^{\!(p)} v =0$ on $U$ belongs to $\CC^{1, \beta}(K)$. Moreover, there is a positive constant $\Cn=\Cn(n,p,\d(K,\partial U), \Lambda_{0,\alpha}, \lambda)$ such that
\begin{align}
\abs{\D v(x)} \leq \Cn \norm{v}_{L^\infty(U)} &&\abs{\D v(x)-\D v(y)} \leq \Cn \norm{v}_{L^\infty(U)}\,\left(\vphantom{\sum} \d(x,y)\right)^\beta
\end{align}
for any $x,y\in K$.
\end{theorem}

On the other hand, the classical regularity theory for quasilinear nondegenerate elliptic equations ensures that Sobolev functions satisfying \eqref{def-p-harm} are smooth around the points where the gradient does not vanish (see \cite[Chapter 4, Section 6]{Ladyzhenskaia1968}). Moreover, the classical elliptic regularity theory can be applied to get higher order interior estimates (see \cite[Theorem 6.6]{Gilbarg2015}).

\begin{theorem}[Higher order Schauder estimates]\label{thm:moreregularityp-potential}
Let $(M,g)$ be a Riemannian manifold of dimension $n$, $U \subseteq M$ be an open smooth subset of $M$ and consider  $1<p<n$. Then for any $k \in\N$, $\beta \in (0,1)$ and $K\Subset U$, and every bounded solution $v :U \to \R$ of the problem $\Delta_g^{\!(p)} v =0$ on $U$ such that $\abs{ \D v} \geq m > 0$ on $K$, there exists a constant $\Cn_{k,\beta} =\Cn_{k,\beta} (n,p,  \d(K, \partial U), \Lambda_{k-1,\beta}, \lambda,m)$ such that
\begin{align}
\norm{v}_{\CC^{k,\beta}(K)}\leq \Cn_{k,\beta}\norm{v}_{L^\infty(U)}.
\end{align}
\end{theorem}

Given $U\subseteq M$ with Lipschitz boundary, a $p$-harmonic function $u\in W^{1,p}(U)$ attains some Dirichlet data $g\in L^p(\partial U)$ if $u$ coincides with $g$ on $\partial U$ in the sense of the trace operator. We report in the next remark the issue of the boundary regularity.
\begin{remark}[Boundary regularity of $p$-harmonic functions]
\label{boundary-regularity}
We point out that, if a $p$-harmonic function attains some $\mathscr{C}^{1, \alpha}$-Dirichlet data on a $\mathscr{C}^{1,\alpha}$ boundary, then the $\mathscr{C}^{1, \beta}$-estimates of \cref{thm:regularityp-potential} can be extended up to the boundary. This is a major contribution of \cite{Lieberman1988}. Moreover, if its gradient does not vanish at the boundary the function is smooth up to the boundary and \cref{thm:moreregularityp-potential} extends as well.
\end{remark}



We finally retrieve the Comparison Principles \cite[Lemma 3.1, Proposition 3.3.2]{Tolksdorf1983} by Tolksdorf, specialised for our purposes. 


\begin{theorem}[Comparison Principles]\label{thm:comparison_principle}
Let $(M,g)$ be a complete Riemannian manifold, $U\subseteq M$ be an open bounded subset and $v_1, v_2:U \to \R$ be two $p$-harmonic functions.
\begin{itemize}[label={\tiny\raisebox{.7ex}{\textbullet}}, leftmargin=.3in]
\item \emph{(Weak) Comparison Principle.} If $v_1,v_2 \in \CC^0(\overline{U})$ and $v_1 \leq v_2$ on $\partial U$, then $v_1 \leq v_2$ on $U$.
\item \emph{Strong Comparison Principle.} Suppose in addition that $U $ is connected, $v_1 \in \CC^1(\overline{U})$, $v_2\in \CC^2(\overline{U})$ and $\abs{ \nabla v_2}\geq \delta >0$ in $U$. If $v_1 \leq v_2$ (resp. $v_1 \geq v_2$) on $U$, then $v_1 =v_2$ or $v_1 <v_2$ (resp. $v_1 >v_2$) on $U$.
\end{itemize}
\end{theorem}

To conclude, we want to recall a compactness theorem that holds for $p$-harmonic functions. It is a natural question whether the limit of a sequence of $p$-harmonic functions is still $p$-harmonic. Clearly the weak formulation in \eqref{def-p-harm} suggests that $\CC^1$-convergence on compact subsets is enough to ensure that also the limit function is $p$-harmonic. The following theorem relaxes this hypothesis in favour of uniform convergence on compact subsets.

\begin{theorem}[Compactness Theorem] \label{thm:compactness_uniform_convergence_p_harmonic}
    Let $(v_n)_{n\in \N}$ be a sequence of $p$-harmonic functions on $U$ that converges uniformly to $v$ on compact subsets of $U$ as $n \to +\infty$. Then $v \in W^{1,p}_{\loc}$ is $p$-harmonic on $U$. 
\end{theorem}
\proof See \cite[Theorem 3.2]{Heinonen1988}.\endproof

\begin{remark}\label{rmk:compactness_convergence_of_metric}
Suppose that $(U_n)_{n \in \N}$ is a sequence of open subsets converging to $U$ open subset as $n \to +\infty$. Let $g_n$ be a metric on $U_n$ for every $n \in \N$ that locally uniformly converges to some metric $g$ on $U$ as $n \to +\infty$. The above theorem still holds if $v_n$ is $p$-harmonic with respect to the metric $g_n$. Consequently, $v$ is $p$-harmonic on $U$ with respect to $g$.
\end{remark}

\subsection{\texorpdfstring{$p$}{p}-nonparabolic manifolds and the \texorpdfstring{$p$}{p}-capacitary potential}
We analyse here the existence and uniqueness of solution $u_p$ to \eqref{pb-intro} on complete Riemannian manifolds. Given a noncompact Riemannian manifold $M$, we consider the $p$-capacitary potential of a bounded set with smooth boundary $\Omega \subset M$, that is a function $u \in W^{1, p}_{\mathrm{loc}}(M \smallsetminus \Omega)$ solving \eqref{pb-intro}. The regularity results previously discussed ensure that $u$ belongs to $\CC^{1,\beta}_{\mathrm{loc}}(M\smallsetminus \Omega)$ and it is smooth near the points where the gradient does not vanish. In particular, by Hopf Maximum Principle \cite{Tolksdorf1983} the datum on $\partial \Omega$ is attained smoothly.

We now focus on some classical sufficient conditions to ensure the existence of the $p$-capacitary potential, which turns out to be related to the notion of $p$-Green's function we are going to recall.

\begin{definition}[$p$-Green's function]
Let $(M, g)$ be a complete Riemannian manifold. Let $\mathrm{Diag}(M) = \{(x, x) \in M \times M \, \vert \, x \in M\}$. For $p \geq 1$, we say that $G_p : M \times M \smallsetminus \mathrm{Diag}(M) \to \R$ is a \emph{$p$-Green's function} for $M$ if
it weakly satisfies $\Delta_p G(o, \,\cdot\,) = -\delta_o$ for any $o \in M$, where $\delta_o$ is the Dirac delta centred at $o$, that is, if it holds
\begin{equation}
\label{green-def}
\int\limits_M \Big\langle \abs{\D \, G_p{(o,\, \cdot\,)}}^{p-2} \, \D \,G_p{(o, \,\cdot\,)} \, \Big\vert \, \D \psi\Big\rangle \dd\mu =  \psi (o)
\end{equation}
for any $\psi \in \CC^{\infty}_c (M)$.
\end{definition}
The notion of $p$-Green's function immediately calls for that of $p$-nonparabolic Riemannian manifold.

\begin{definition}[$p$-nonparabolicity]\label{def:pnop}
We say that a complete noncompact Riemannian manifold $(M, g)$ is $p$-nonparabolic if, for any $o \in M$, there exists a \emph{positive} $p$-Green's function $G_p : M  \smallsetminus \{o\} \to \R$. With the expression \emph{$p$-Green function} we are in fact referring to the positive minimal one.
\end{definition} 
The notion of $p$-nonparabolicity is intimately related to existence of a solution to \eqref{pb-intro}, in that if the $p$-Green's function of a $p$-nonparabolic Riemannian manifold vanishes at infinity, then such solution exists for any open bounded subset $\Omega \subset M$ with smooth boundary.  A complete and self contained proof of this fact is provided in the Appendix of \cite{Fogagnolo2020a}. We report the statement of such basic though fundamental result.
\begin{theorem}[Existence of the $p$-capacitary potential]\label{thm:existence_of_p_potential}
Let $(M, g)$ be a complete noncompact $p$-nonparabolic Riemannian manifold. Assume also that the $p$-Green's function $G_p$ satisfies $G_p(o, x) \to 0$ as $\d(o, x) \to +\infty$ for some $o \in M$. Let $\Omega \subset M$ be an open bounded subset with smooth boundary. Then, there exists a unique solution $u_p$ to \eqref{pb-intro}. 
\end{theorem}
We want to underline that, the existence of a $p$-Green's function does not guarantee that it vanishes at infinity. This last property is related to the geometry of all ends. We refer the reader to \cite{Holopainen1990,Holopainen1999} for a detailed discussion on this topic.

It is convenient to recall here the definition of $p$-capacity of an open bounded subset $\Omega \subset M$ together with a normalised version of it which turns out to be more convenient for our computations.
\begin{definition}[$p$-capacity and normalised $p$-capacity]\label{def:capandnormpcap}
Let $(M, g)$ be a complete noncompact Riemannian manifold, and let $\Omega$ be an open bounded subset of $M$.
The \emph{$p$-capacity} of $\Omega$ is defined as
\begin{equation}
\label{cap}
\mathrm{Cap}_p (\Omega) = \inf \set{\,\int\limits_M \abs{\D v}^p \dd \mu\st v \in \CC_c^\infty(M),\, v \geq 1\text{ on } \Omega}.
\end{equation}
On the other hand, the \emph{normalised $p$-capacity} of $\Omega$ is defined as 
\begin{equation}\label{eq:normalisedp-cap}
\Cn_p(\Omega)= \inf \set{\left(\frac{p-1}{n-p}\right)^{p-1}\frac{1}{\abs{\S^{n-1}}}\int\limits_M \abs{\D v}^p \dd \mu\st v \in \CC_c^\infty(M),\, v \geq 1\text{ on } \Omega}.
\end{equation}
\end{definition}


A function $u$ solving \eqref{pb-intro} realises the $p$-capacity of the initial set $\Omega$, and actually one can also characterise such quantity with a suitable integral on $\partial \Omega$. We resume these facts in the following statements, whose proof can be found in \cite[Section 2.2]{Benatti2021}

\begin{proposition}
\label{cap-u-prop}
Let $(M, g)$ be a complete noncompact $p$-nonparabolic Riemannian manifold. Assume also that the $p$-Green's function $G_p$ satisfies $G_p(o, x) \to 0$ as $\d(o, x) \to +\infty$ for some $o \in M$. Let $\Omega \subset M$ be an open bounded subset with smooth boundary. Then the solution $u_p$ to \eqref{pb-intro} realises
\begin{equation}
\label{cap-u}
\Cn_p(\Omega) = \left(\frac{p-1}{n-p}\right)^{p-1} \frac{1}{\abs{\S^{n-1}}} \int\limits_{M \smallsetminus \overline{\Omega}} \abs{\D u_p}^p \dd \mu
\end{equation}
Moreover, we have that
\begin{equation}
\label{p-cap-u}
\Cn_p(\Omega) = \left(\frac{p-1}{n-p}\right)^{p-1} \frac{1}{\abs{\S^{n-1}}} \int\limits_{\set{u_p=1/t}} \abs{\D u_p}^{p-1} \dd\sigma.
\end{equation}
holds for almost every $t\in [1,+\infty)$, including any $1/t$ regular value for $u_p$. 
\end{proposition}

In particular, evaluating \eqref{p-cap-u} at $t=1$, which is a regular value by the Hopf Maximum Principle, we have that
\begin{equation}
\label{cap-boundary}
\Cn_p (\Omega) = \left(\frac{p-1}{n-p}\right)^{p-1} \frac{1}{\abs{\S^{n-1}}} \int\limits_{\partial \Omega} \abs{\D u_p}^{p-1} \dd\sigma.
\end{equation}

Moreover, one can actually relate the capacity of $\Omega_t= \set{u>1/t}\cup \Omega$ to the capacity of $\Omega$.

\begin{proposition}\label{prop:scaling-invariant-cap}
Let $(M, g)$ be a complete noncompact $p$-nonparabolic Riemannian manifold, for some $p > 1$. Let $\Omega \subset M$ be an open bounded subset with smooth boundary. Then the solution $u_p$ to \eqref{pb-intro} realises
\begin{equation}\label{eq:scaling-invariant-cap}
    \Cn_p(\Omega_t)= t^{p-1}\Cn_p(\Omega)
\end{equation}
for every $t \in [1,+\infty)$, where $\Omega_t= \set{u>1/t}\cup \Omega$. In particular, the map $t\mapsto \Cn_p(\Omega_t)$ is smooth.
\end{proposition}

\noindent {\em From now on, unless where it is necessary, we fix $1<p$ and we drop the subscript $p$ when we consider a solution $u_p$ to the problem \eqref{pb-intro}.}

Given an end $E\subset M$ and a set $\Omega \subset M$ such that $ \partial E \subset \Omega$ we define the capacity of $\Omega$ restricted to $E$ as 
\begin{equation}
    \capa_p(\Omega;E) = \inf \set{\,\int\limits_E \abs{\D v}^p \dd \mu\st v \in \CC_c^\infty(M),\, v \geq 1\text{ on } \Omega}.
\end{equation}
The set function $\capa_p(\,\cdot\,;E)$ has the same properties of the standard $p$-capacity in \cref{def:capandnormpcap}.

\subsection{Preliminary properties of Asymptotically Conical manifolds}
\label{subsec:prel}

We give here the precise definition of Asymptotically Conical Riemannian manifolds according to \cite{Chodosh2017}. For better comprehension, we recall the definition of the H\"{o}lder seminorm of a tensor field. A tensor field  $T \in \mathcal{T}_s^q(M)$ is $\alpha$-H\"{o}lder continuous at $x$ for some $\alpha\in[0,1]$ if there exists a geodesically convex open neighbourhood $U_x$ centred at $x$ such that 
\begin{align}
\sup_{y\in U_x \smallsetminus \set{x}} \frac{\abs{T(x)-T(y)}_{g}}{\left(\d(x,y)\right)^\alpha}<+\infty
\end{align}
is finite, where, to compute the difference between $T(x)$ and $T(y)$, we parallel transport $T(y)$ onto $x$. The tensor field $T$ is said to be $\alpha$-H\"{o}lder continuous on $U\subset M$ if it is $\alpha$-H\"{o}lder continuous at every $x\in U$. We sometimes omit the subscript $g$ if it is clear the metric we are referring to.

Consider a \textit{cone with link $L$}, namely $((0,+\infty)\times L, \hat{g})$ where $\hat{g}=\dd \rho^2 + \rho^2 g_L$. In this case, let $s>0$ be such that $B_s(x)$ is geodesically convex in $((0,+\infty)\times L, \hat{g})$ for every $x \in\set{1} \times L $. Then, for every $x \in  (0,+\infty)\times L$ the ball of radius $s\rho(x)$ centred at $x$ is still geodesically convex, where $\rho: (0,+\infty)\times L\to (0,+\infty)$ is the projection onto the first coordinate. Given an $\alpha$-H\"{o}lder continuous tensor field $T$, we define the $\alpha$-H\"{o}lder seminorm of $T$ at $x$ as 
\begin{align}
\seminorm{T}^{(s)}_{\alpha,\hat{g}}(x)=\sup_{y\in B_{s\rho(x)}(x) \smallsetminus \set{x}} \frac{\abs{T(x)-T(y)}_{\hat{g}}}{\left(\dist(x,y)\right)^\alpha}.
\end{align}
Observe that, if $T$ is bounded (with respect to $\abs{}_{\hat{g}}$) and $s,t>0$ satisfy the above assumptions, $\seminorm{T}^{(s)}_{\alpha,\hat{g}}(x)=\seminorm{T}^{(t)}_{\alpha,\hat{g}}(x)$ for any $x \in (R,+\infty)\times L$ provided $R$ is large enough. Then, the following definition is well-posed and we can drop the superscript $(s)$.

\begin{definition}[$\CC^{k,\alpha}$-Asymptotically Conical Riemannian manifolds] \label{def:C0ACintro}
Let $(M,g)$ be a Riemannian manifold, $k \in \N$ and $\alpha\in [0,1)$. $M$ is said to be \emph{$\CC^{k,\alpha}$-Asymptotically Conical}  if there exists an open bounded subset $K\subseteq M$, a closed smooth hypersurface $L$ and a diffeomorphism $\pi:M\smallsetminus K\to [1,+\infty)\times L$ such that
\begin{equation}
\label{eq:C0AC-topology}
\sum_{i=0}^k\rho^{i}\abs{\D^{(i)}_{\hat{g}} (\pi_* g - \hat{g})}_{\hat{g}}+ \rho^{k+\alpha}\seminorm{\D^{(k)}_{\hat{g}} (\pi_* g - \hat{g})}_{\alpha,\hat{g}}=o(1) ,
\end{equation}
as $\rho \to +\infty$, where $\rho:[1,+\infty) \times L \to [1,+\infty)$ is the projection map onto the first component and $\hat{g}=\dd \rho^2 +\rho^2 g_L$ is the cone metric. We use the convention $\CC^{k}= \CC^{k,0}$.
\end{definition}

\cref{def:C0ACintro} says that in a $\CC^{k,\alpha}$-Asymptotically Conical Manifold the metric $g$ approaches the metric $\hat{g}$ of a \textit{truncated cone with link $L$} with respect to a scaling invariant $\CC^{k,\alpha}$-norm. The diffeomorphism $\pi:M \smallsetminus K \to [1,+\infty)\times L$ identifies the boundary of $K$ with the link $L$. With abuse of notation, $\pi_*\rho:M\smallsetminus K\to [1,+\infty)$ will be denoted by $\rho$ and $\pi_*\hat{g}=\dd \rho ^2 + \rho^2 g^{}_{L}$ by $\hat{g}$. Moreover, by convention the set $\set{\rho<1}$ is used to denote $K$ and accordingly $\set{\rho \leq r} =M \smallsetminus \set{\rho>r}$ and $\set{1 \leq \rho \leq r} = M \smallsetminus ( \set{\rho>r}\cup K)$. Given any coordinate system $(\vartheta^1, \ldots, \vartheta^{n-1})$ on an open subset $U$ of $L$, $(\rho, \vartheta^1, \ldots, \vartheta^{n-1})$ are coordinates on $(1,+\infty) \times U\subset M\smallsetminus \overline{K}$. The condition $\abs{g- \hat{g}}_{\hat{g}} = o(1)$ as $\rho \to +\infty$ is equivalent to a condition on the coordinates that can be read as
\begin{align}
    g_{\rho \rho}=1+o(1)&& g_{\rho j}=o(\rho) && g_{ij}= \rho^2 g^L_{ij}+ o (\rho^2)
\end{align}
for every $i,j =1, \ldots,n-1$ as $\rho \to +\infty$. By using Cramer's rule to solve the system and Laplace expansion to compute determinants, we obtain 
\begin{align}
    g^{\rho \rho}=1+o(1)&& g^{\rho j}=o(\rho^{-1}) && g^{ij}= {\rho^{-2}} g_L^{ij}+ o (\rho^{-2}).
\end{align}
The $\CC^{0,\alpha}$-Asymptotically Conical condition for $\alpha>0$ gives, in addition, information on the H\"{o}lder seminorm of the components. Indeed, arguing as before we get that 
\begin{align}
  \seminorm{g_{\rho\rho}-1}_{\alpha,\hat{g}} = o(\rho^{-\alpha})&&\seminorm{g_{\rho j}}_{\alpha,\hat{g}}=o(\rho^{1-\alpha})&&\seminorm{g_{ij} - \rho^2 g^L_{ij}}_{\alpha,\hat{g}} = o(\rho^{2-\alpha})
\end{align}
for every $i,j=1,\ldots, n-1$ as $\rho \to +\infty$. Increasing $k$ in the $\CC^{k,\alpha}$-Asymptotically Conical assumption we gain knowledge about the $k$-th derivative of the components of $g$. Increasing $k$ in the $\CC^{k,\alpha}$-Asymptotically Conical assumption we gain knowledge about the $k$-th derivative of $(g-\hat{g})$. 

Consider for every $s > 0$ the family of diffeomorphism on $(0,+\infty) \times L$ defined as
\begin{equation}\label{eq:dilatation}
\begin{array}{rlcl}
\omega_s:&(0,+\infty)\times L&\longrightarrow&(0,+\infty) \times L\\[.2cm]
&(\rho, \vartheta^1,\ldots,\vartheta^{n-1}) &\longmapsto&(s\rho,\vartheta^1,\ldots,\vartheta^{n-1}),
\end{array}
\end{equation}
With abuse of language we will also denote by $\omega_s$ any restriction of it to some truncated cone. Since $\omega_s$ induces a family of diffeomorphisms from $(1/s, +\infty) \times L$ onto $\set{\rho \geq 1} \subset M$ through the composition with $\pi$ in \cref{def:C0ACintro}, we will also denote by $\omega_s$ such map. The condition \eqref{eq:C0AC-topology} can be also interpreted as the convergence of the family of metrics on the cone $(0,+\infty)\times L$, built for every $s\geq 1$ by pulling the metric $g$ back through the diffeomorphism $\omega_s$ and properly rescaling them. This is the content of the following lemma.

\begin{lemma}\label{lem:convergence}
A complete Riemannian manifold $(M,g)$ is $\CC^{k,\alpha}$-Asymptotically Conical if and only if the metric $g_{(s)}=s^{-2}\omega_s^*g$ satisfies
\begin{equation}
\sum_{i=0}^k\rho^i\abs{\D^{(i)}_{\hat{g}} (g_{(s)} - \hat{g})}_{\hat{g}}+ \rho^{k+\alpha}\seminorm{\D^{(k)}_{\hat{g}} (g_{(s)} - \hat{g})}_{\alpha,\hat{g}} =o(1),
\end{equation}
as $s\to +\infty$ on $[R,+\infty) \times L$ for every $R>0$.
\end{lemma}

\proof Since $\omega^*_s\dd \rho = r \dd \rho$ and $\omega^*_s\dd \vartheta^i= \dd \vartheta^i$, is clear that $s^{-2} \omega^*_s \hat{g} = \hat{g}$. Thus the the case of $\CC^0$-Asymptotically Conical manifold follows from algebra operations on tensors. The result for $k\in \N$ and $\alpha \in [0,1]$ follows in the same way from the fact that $\D_{\hat{g}}s^{-2}\omega^*_sg=s^{-2} \omega^*_s(\D_{\hat{g}} g)$ and $\dist_{\hat{g}}(x,y)=\dist_{\hat{g}} (\omega_s(x),\omega_s(y))/s$ for every $x,y \in (1/s,+\infty)\times L$.\endproof

We want to highlight the relation between the coordinate $\rho$ and the distance induced by $g$ on $M$.

\begin{lemma}\label{lem:asymptoticaldistance}
Let $(M,g)$ be a complete $\CC^0$-Asymptotically Conical Riemannian manifold and $o \in M$. Then
\begin{equation}\label{eq:asymptoticaldistance}
    \lim_{\d(o,x)\to +\infty} \frac{\d(o,x)}{\rho(x)}=1.
\end{equation}
\end{lemma}

Observe that since $K$ is compact there exist a $R>0$ such that $\d(o,x)>R$ implies $x\in M\smallsetminus K$, hence \eqref{eq:asymptoticaldistance} makes sense. Since $\pi$ is a diffeomorphism, taking the limit for $\d(o,x)\to+ \infty$ is the same of taking it for $\rho(x)\to +\infty$.

\proof Since $\abs{\D \rho }_g=1+o(1)$, for every $\varepsilon>0$ we can find $R_\varepsilon>1$ such that $1- \varepsilon \leq \abs{\D \rho}_g\leq 1+\varepsilon$ on $\set{\rho\geq R_\varepsilon}$. Pick $x \in \set{\rho\geq R_\varepsilon}$ and a curve $\gamma:[R_\varepsilon,\rho(x)]\to M$ which is the solution to the problem
\begin{equation}
    \begin{cases}
    \dot{\gamma}(s)= \frac{\D \rho}{\abs{\D \rho}^2}(\gamma(s)),\\
    \gamma(\rho(x))= x.
    \end{cases}
\end{equation}
Computing the length of $\gamma$ we get
\begin{align}
    L(\gamma) &= \int^{\rho(x)}_{R_\varepsilon}\abs{\dot{\gamma}(s) }\dd s =\int^{\rho(x)}_{R_\varepsilon}\frac{1}{\abs{\D \rho}_g}(\gamma(s)) \dd s \leq\frac{{\rho(x)}-R_\varepsilon}{1-\varepsilon},
\end{align}
which ensures that
\begin{align}
    \limsup_{\rho(x)\to+\infty}\frac{\d(o,x)}{\rho(x)}&\leq \limsup_{\rho(x)\to +\infty}\frac{L(\gamma)+2\diam( \set{\rho\leq R_\varepsilon})}{\rho(x)}\leq \frac{1}{1-\varepsilon}.
\end{align}

Conversely, consider any geodesic $\sigma:[0,L]\to M$, parametrised by arc length, joining $\sigma(0)\in\set{\rho =R_\varepsilon}$ and $\sigma(L)=x$. Then we obtain
\begin{align}
    \rho(x)-R_\varepsilon &= \int_0^{L}\ip{\D \rho, \dot{\sigma}(s)} \dd s  \leq \int_0^{L} \abs{\D \rho}_g(\sigma(s)) \dd s\leq(1+ \varepsilon)L
\end{align}
which yields
\begin{align}
    \liminf_{\rho(x)\to +\infty}\frac{\d(x,o)}{\rho(x)}&\geq \liminf_{\rho(x)\to +\infty}\frac{L- 2\diam( \set{\rho\leq R_\varepsilon})}{\rho(x)}\\&\geq \liminf_{\rho(x)\to +\infty} \frac{\rho(x)-R_\varepsilon-2\diam( \set{\rho\leq R_\varepsilon})}{(1+{\varepsilon})\rho(x)}=\frac{1}{1+{\varepsilon}}.
\end{align}
By the arbitrariness of $\varepsilon>0$ we can conclude.
\endproof

In Riemannian manifolds with nonnegative Ricci curvature, in virtue of Bishop-Gromov theorem, one can define an Asymptotic Volume Ratio since \begin{equation}
    \AVR(g)=\lim_{r \to +\infty} \frac{ \abs{B(o,R)}}{ \abs{\B^n}R^n}
\end{equation}
exists and does not depend on $o\in M$. Here we relaxed the condition on Ricci curvature so we can not apply Bishop-Gromov theorem, but on the other side we require an asymptotic behaviour for the metric, that allows to define an Asymptotic Volume Ratio as well.

\begin{lemma}\label{lem:AVRconico}
Let $(M,g)$ be a complete $\CC^0$-Asymptotically Conical Riemannian manifold. Then
\begin{equation}\label{eq:AVRconico}
    \frac{ \abs{L}}{\abs{\S^{n-1}}}=\lim_{R\to +\infty}\frac{\abs{ \set{1 \leq \rho \leq R}}}{\abs{\B^{n}}R^n}= \lim_{R \to + \infty} \frac{ \abs{\set{\rho =R}}}{\abs{\S^{n-1}} R^{n-1}}
\end{equation}
where $L$ is the link of the cone $(M,g)$ is asymptotic to.
\end{lemma}

\proof One can easily show that $\det(g)= \det(\hat{g})(1+o(1))=\rho^{2(n-1)}\det(g_L)(1+o(1))$. Hence, for every $\varepsilon>0$ there exists $R_\varepsilon\geq 1 $ such that
\begin{equation}
    \abs{L}(1-\varepsilon)\frac{(R^n-R_\varepsilon^n)}{n}\leq \abs{\set{R_\varepsilon \leq \rho \leq R}}\leq \abs{L}(1+ \varepsilon)\frac{(R^n-R_\varepsilon^n)}{n}
\end{equation}
Dividing each term by $\abs{\B^n}R^n $ one get
\begin{equation}
    \abs{L}(1-\varepsilon) \frac{(R^n-R_\varepsilon^n)}{n\abs{\B^n}R^n}\leq \frac{\abs{ \set{ R_\varepsilon\leq \rho \leq R}}}{\abs{\B^n}R^n}\leq \abs{L}(1+\varepsilon) \frac{(R^n-R_\varepsilon^n)}{n\abs{\B^n}R^n}.
\end{equation}
Since $\abs{\set{ 1\leq \rho\leq R_\varepsilon}}/(\abs{ \B^n}R^n) $ vanishes as $R\to +\infty$ we obtain
\begin{equation}
    (1-\varepsilon) \frac{\abs{L}}{n \abs{ \B^n}} \leq \lim_{R \to +\infty}\frac{\abs{ \set{ 1\leq \rho \leq R}}}{\abs{\B^n}R^n}\leq (1+\varepsilon) \frac{\abs{L}}{n \abs{ \B^n}}
\end{equation}
which in turn gives the first identity in \eqref{eq:AVRconico} by arbitrariness of $\varepsilon$. We turn to prove the second identity. Since De L'H\^{o}pital rule gives
\begin{equation}
    \lim_{R\to +\infty}\frac{ \frac{\dd}{\dd R} \abs{ \set{1 \leq \rho \leq R}}}{\abs{ \S^{n-1}}R^{n-1}}= \lim_{R\to +\infty}\frac{ \abs{\set{1 \leq \rho \leq R}}}{\abs{ \B^n}R^n}= \frac{\abs{L}}{\abs{\S^{n-1}}}
\end{equation}
and
\begin{equation}
    \frac{\dd}{\dd R}\abs{ \set{1 \leq \rho \leq R}} = \frac{\dd}{\dd R}\int\limits_1^R\int\limits_{\set{\rho =s}} \frac{1}{\abs{ \D \rho}} \dd \sigma \dd s = \abs{\set{\rho=R}}(1+ o(1)),
\end{equation}
we conclude the proof. \endproof

Coupling this result with \cref{lem:asymptoticaldistance} one gets that 
\begin{equation}
    \lim_{R \to +\infty}\frac{\abs{B(o,R)}}{R^n \abs{ \B^n}} =  \lim_{R\to +\infty}\frac{\abs{ \set{1 \leq \rho \leq R}}}{\abs{\B^{n}}R^n}= \frac{ \abs{L}}{\abs{\S^{n-1}}}
\end{equation}
for every $o \in M$. Hence the left hand side limit exists and does not depend on the point $o\in M$. We can finely give the following definition.

\begin{definition}\label{def:AVRconico}
Let $(M,g)$ be a complete $\CC^0$-Asymptotically Conical Riemannian manifold. The Asymptotic Volume Ratio of $(M,g)$ is defined as
\begin{equation}
    \AVR(g)= \frac{\abs{L}}{\abs{\S^{n-1}}}
\end{equation}
where $L$ is the link of the cone $(M,g)$ is asymptotic to.
\end{definition}

In this case $0<\AVR(g)$, but in general $\AVR(g)$ could exceed $1$ and $\AVR(g)=1$ does not imply that $(M,g)$ is isometric to the flat Euclidean space.

As already observed in the Introduction, a complete $\CC^0$-Asymptotically Conical Riemannian manifold is forced to have a finite number of ends.

\begin{lemma}\label{lem:finiteends}
A complete $\CC^0$-Asymptotically Conical Riemannian manifold $(M,g)$ has finitely many ends with connected boundary each of them is differomorphic to $[1,+\infty) \times L_i $ where $L_i$ is a connected component of the link of the asymptotic cone.
\end{lemma}

\proof Since $L$ is a compact hypersurface, it has a finite number of connected component. Each end with respect to $K$ is therefore diffeomorphic to a cone on a connected component of $L$. \endproof

As already mentioned in the Introduction, given $E_1, \ldots, E_N$ the ends of $M$, each $E_i$ is $\CC^0$-Asymptotically Conical and we can define the Asymptotic Volume ratio of $E_i$ as 
\begin{equation}
\label{eq:avrendtext}
    \AVR(g; E_i)= \lim_{R\to +\infty}\frac{\abs{B(o,R)\cap E_i}}{ \abs{\mathbb{B}^n}R^n} = \frac{ \abs{L_i}}{\abs{\S^{n-1}}},
\end{equation}
for every $i=1,\ldots, N$, where as a above $L_i$ denotes a connected component of the link of the asymptotic cone. Moreover, the Asymptotic Volume Ratio of $(M,g)$ splits as 
\begin{equation}\label{eq:avrsumtext}
    \AVR(g)= \sum_{i=1}^N\AVR(g;E_i)
\end{equation}
where $\AVR(g)>0$ is the Asymptotic Volume Ratio of $(M,g)$.

If $\Ric_L\geq (n-2) g_L$, then the cone $([1,+\infty) \times L, \hat{g})$ has nonnegative Ricci curvature and in particular $\AVR(g;E_i)\leq 1$ for every $i=1, \ldots, k$. This condition is automatically true if the manifold is $\CC^{0}$-Asymptotically Conical and $\Ric$ satisfies \eqref{eq:Ricci-intro}, thanks to the following lemma.

\begin{lemma}\label{prop:Ricci-limit-cone}
Let $(M,g)$ be a complete $\CC^0$-Asymptotically Conical Riemannian manifold. Suppose that $\Ric_g\geq -f(\dist(x,o))$ for some nonnegative smooth function $f(t)=o(1)$ as $t \to +\infty$, for some $o \in M$. Then $\Ric_{\hat{g}}\geq 0$, where $\hat{g}= \dd \rho^2 + \rho^2 g_L$ is the asymptotical conic metric of $g$ and $L$ is the link of the limit cone. In particular, $\AVR(g)\leq N$ where $N$ is the number of connected component of the link $L$.
\end{lemma}

\proof By \cref{lem:convergence} we can assume that $f$ is a function of $\rho$. Let $\omega_s$ be as in \eqref{eq:dilatation}, we denote by $g_{(s)}$ the metric $s^{-2}\omega_s^* g$ on $[1/s,+\infty) \times L$. It is easy to prove that $( [1/s,+\infty) \times L, \dist_{g_{(s)}}, x)$ converges in the pointed-Gromov-Hausdorff topology to $([0,+\infty) \times L, \dist_{\hat{g}},x)$ for some $x \in \set{\rho =2}$. Since 
\begin{equation}
    \lim_{s \to +\infty} \abs{ B(x,1)}_{g_{(s)}} = \abs{B(x,1)}_{\hat{g}}
\end{equation}
by \cite[Theorem 1.2]{DePhilippis2018}  $( [1/s,+\infty) \times L, \dist_{g_{(s)}}, \mu_{g_{(s)}}, x)$ converges to $( [0,+\infty)\times L, \dist_{\hat{g}},\mu_{\hat{g}},x)$ in the pointed-measured-Gromov-Hausdorff topology. By \cite[Theorem 7.2]{Gigli2015b} $\Ric_{\hat{g}}\geq -f(s)$ for every $s$, hence $\Ric_{\hat{g}}\geq 0$. In particular, $\abs{L_i}\leq \abs{ \S^{n-1}}$, hence $\AVR(g)\leq N$ by \cref{eq:avrsumtext,eq:avrendtext}.\endproof

In \cite[Theorem 1.7]{Mari2019} the authors guarantee the existence of the (weak) IMCF starting at $\Omega \subseteq M$ with smooth boundary whenever the Ricci curvature satisfies a  nondecreasing lower bound and a global $L^1$-Sobolev Inequality is in force, that is 
\begin{equation}\label{eq:L1Sobolev}
    \left(\;\int\limits_M \abs{\varphi}^{\frac{n}{n-1}} \dd \mu\right)^{\frac{n-1}{n}} \leq \Cn_{S}\int\limits_M \abs{\D \varphi} \dd \mu
\end{equation}
for every $\phi \in \Lip_c(M)$ where $\Cn_S$ is some finite constant depending only on the geometry of the manifold.

It is well known that the existence of a finite constant $\Cn_S$ for the $L^1$-Sobolev Inequality is equivalent to the existence of a positive Isoperimetric constant, which is a $\Cn_{\rm Iso}>0$ such that
\begin{equation}\label{eq:Isoperimetric}
    \Cn_{\rm Iso}  \leq \frac{\abs{\partial K}^n}{\abs{K}^{n-1}}
\end{equation}
for every compact domain $K$ (see \cite[Remark 6.6]{Federer1960} or \cite[pp. 89-90]{Schoen1994}).

Riemannian manifolds with nonnegative Ricci curvature and Euclidean Volume Growth satisfies both \eqref{eq:L1Sobolev} and \eqref{eq:Isoperimetric} as observed by \cite{Varopoulos1985} (see also \cite{Carron1994} and \cite[Theorem 8.4]{Hebey1999}). It is then plausible that these inequalities hold also when the manifold asymptotically behaves as a cone with nonnegative Ricci curvature.

\begin{proposition}\label{prop:Sobolev_in_AC}
Let $(M,g)$ be a $\CC^0$-Asymptotically Conical Riemannian manifold. Suppose that $\Ric_g\geq - f(\dist(x,o))$ for some nonnegative smooth function $f(t)=o(1)$ as $t \to +\infty$, for some $o \in M$. Then $(M,g)$ admits a global $L^1$-Sobolev Inequality \eqref{eq:L1Sobolev} for some finite constant $\Cn_S$ or equivalently it has a positive isoperimetric constant $\Cn_{\rm Iso}>0$.
\end{proposition}

\proof In virtue of \cite[Theorem 3.2]{Pigola2014} it is enough to prove that a $L^1$-Sobolev Inequality is satisfied outside some compact set. By \cref{lem:finiteends} $(M,g)$ has only a finite number of ends each of them corresponding to one connected component of the link $L$. Thus, we can assume that $(M,g)$ has only one Asymptotically Conical end $E$. By \cref{prop:Ricci-limit-cone}, $E$ asymptotically behaves as a cone with nonnegative Ricci curvature, that satisfies the $L^1$-Sobolev Inequality for some constant $\hat{C}$. Suppose by contradiction that for every compact $K\subset E$, the $L^1$-Sobolev Inequality is not satisfied on $E \smallsetminus K$. Since the metric $g$ converges to the metric $\hat{g}= \dd \rho ^2 + \rho^2 g_L$, for every $\varepsilon>0$ there exists a compact $K_\varepsilon$ such that
\begin{align}
    \abs{\;\int\limits_M \abs{ \varphi}^{\frac{n}{n-1}}\dd { \mu_g}-\int\limits_M \abs{ \varphi}^{\frac{n}{n-1}}\dd { \mu_{\hat{g}}}}\leq \varepsilon \int\limits_M \abs{ \varphi}^{\frac{n}{n-1}}\dd { \mu_{\hat{g}}}
    \intertext{and}
        \abs{\; \int\limits_M \abs{ \D\varphi}_g\dd { \mu_g}-\int\limits_M \abs{\D \varphi}_{\hat{g}}\dd { \mu_{\hat{g}}}}\leq \varepsilon \int\limits_M \abs{\D \varphi}_{\hat{g}}\dd { \mu_{\hat{g}}}
\end{align}
for every $\varphi\in \Lip_c(E \smallsetminus K_\varepsilon)$. Moreover, for every $\Cn$ there exists a function $\varphi \in \Lip_c(E \smallsetminus K_\varepsilon)$ such that
\begin{equation}
        \left(\;\int\limits_M \abs{\varphi }_g^{\frac{n}{n-1}} \dd {\mu_g}\right)^{\frac{n-1}{n}} > \Cn \int\limits_M \abs{\D \varphi}_g \dd {\mu_g}
\end{equation}
Then for every $\varepsilon<1$ and $\Cn$ we have $\varphi\in \Lip_c(E \smallsetminus K_\varepsilon)$ that satisfies
\begin{align}
    \left(\; \int\limits_M \abs{ \varphi}_{\hat{g}}^{\frac{n}{n-1}} \dd { \mu_{\hat{g}}}\right)^{\frac{n-1}{n}}&\geq\frac{1}{ (1+ \varepsilon)^{ \frac{n-1}{n}}} \left(\; \int\limits_M \abs{ \varphi}_{g}^{\frac{n}{n-1}} \dd { \mu_{g}}\right)^{\frac{n-1}{n}} >\frac{\Cn}{ (1+ \varepsilon)^{ \frac{n-1}{n}}} \int\limits_M \abs{\D \varphi}_g \dd {\mu_g}\\
    & \geq \Cn\,\frac{ (1-\varepsilon)}{ (1+ \varepsilon)^{ \frac{n-1}{n}}}\int\limits_M \abs{\D \varphi}_{\hat{g}} \dd {\mu_{\hat{g}}}.
\end{align}
It is enough to choose $\varepsilon<1$ and $\Cn$ such that
\begin{equation}
   \Cn\, \frac{(1-\varepsilon)}{ (1+ \varepsilon)^{ \frac{n-1}{n}}}> \hat{\Cn}
\end{equation}
to obtain a contradiction to the $L^1$-Sobolev Inequality on the asymptotic cone.\endproof

\smallskip

With some analogy,  considering $\Omega\subseteq M$ some open subset with smooth boundary and $u_p:M\smallsetminus \Omega \to \R$ the $p$-capacitary potential associated to $\Omega$, $1 <p <n $, we recall from the Introduction that
\begin{equation}\label{eq:ends_capacity_decomposition}
    \Cn^{(i)}_p(\Omega) =\left(\frac{p-1}{n-p}\right)^{p-1} \frac{1}{\abs{\S^{n-1}}} \int\limits_{\set{u_p=1/t}\cap E_i} \abs{\D u_p}^{p-1} \dd\sigma = \frac{\Cn_p\left(\set{u\leq \frac{1}{t}} \cap E_i; E_i\right) }{t^{p-1}}.
\end{equation}
This is well defined since for every $T \in [1,+\infty)$ large enough, $\set{ u_p>1/T}$ contains the compact $K$ in \cref{def:C0ACintro} and the quantity in \cref{eq:ends_capacity_decomposition} does not depend on $t \geq T$. Moreover, it is readily checked that, by \eqref{p-cap-u}, the $p$-capacity of $\Omega$ splits as
\begin{equation}
    \Cn_p(\Omega)= \sum_{i=1}^{N} \Cn^{(i)}_p(\Omega).
\end{equation}
The quantity $\Cn^{(i)}_p(\Omega)$ represents the portion of $\Omega$ that contributes to its $p$-capacity under the influence of the end $E_i$. In the case $\Omega$ is already contains $K$ in \cref{def:AVRconico} it is exactly the $p$-capacity of $\Omega \cap E_i$.

On cones, the $p$-capacity of the cross section $\set{\rho=r}$ can be easily computed since the function $u=(\rho/r)^{-(n-p)/(p-1)}$ is the $p$-capacitary potential associated to these sets. In Asymptotically Conical Riemannian manifold one might expect that the $p$-capacity of $\set{\rho =r}$ approaches the model one for large $r$. Despite the definition of the $p$-capacity involves the first order derivatives of the $p$-capacitary potential, the convergence is also true even if the metric converges only in the $\CC^0$-topology.

\begin{lemma}\label{lem:limit_p_cap_cross_sections}
Let $(M,g)$ be a $\CC^0$-Asymptotically Conical Riemannian manifold. Let $\rho$ be the radial coordinate. Then
\begin{equation}\label{eq:limit_p_cap_cross_sections}
    \lim_{r \to +\infty}\frac{ \capa_p\left(\set{\rho \leq r}\right)}{r^{n-p}\abs{\S^{n-1}}}= \left( \frac{n-p}{p-1}\right)^{p-1} \AVR(g).
\end{equation}
\end{lemma}

\proof Since the metric $g$ converges to the metric $\hat{g}$, for every $\varepsilon>0$ there exists a $R_\varepsilon>0$ such that for every $r \geq R_\varepsilon$ 
\begin{equation}
    \abs{\; \int\limits_{M} \abs{\D \varphi}_g^p \dd{\mu_g} - \int\limits_{M} \abs{\D \varphi}_{\hat{g}}^p \dd{\mu_{\hat{g}}}} \leq \varepsilon   \int\limits_{M} \abs{\D \varphi}_{\hat{g}}^p \dd{\mu_{\hat{g}}}
\end{equation}
holds for every function $\varphi \in \CC^\infty_c(\set{\rho\geq r})$ such that $\varphi=1$ on $\set{\rho =r}$. In particular, we have that 
\begin{equation}
    (1-\varepsilon) \int\limits_{M} \abs{\D \varphi}_{\hat{g}}^p \dd{\mu_{\hat{g}}} \leq  \int\limits_{M} \abs{\D \varphi}_g^p \dd{\mu_g}  \leq (1+\varepsilon)\int\limits_{M} \abs{\D \varphi}_{\hat{g}}^p \dd{\mu_{\hat{g}}}.
\end{equation}
The set $\set{\rho \geq r}$ is diffeomorphic to $[r, +\infty)\times L$ where $L$ is the cross section of the cone $(M,g)$ is asymptotic to. Hence, the family of $\varphi$ considered above are in one-to-one correspondence with the competitors for the $p$-capacity of $\set{\rho \leq r}$ in the Riemannian cone $[r, +\infty)\times L$. Dividing each side by $\abs{\S^{n-1}}$, recalling the characterisation of $\AVR(g)$ in \cref{def:AVRconico} and taking the infimum on each side of the previous chain of inequalities we are left with
\begin{equation}
    (1- \varepsilon) r^{n-p} \AVR(g)\left( \frac{n-p}{p-1} \right)^{p-1} \leq \frac{\capa_p(\set{\rho \leq r})}{\abs{\S^{n-1}}} \leq (1+\varepsilon) r^{n-p} \AVR(g)\left( \frac{n-p}{p-1} \right)^{p-1}.
\end{equation}
dividing each term by $r^{n-p}$ and sending to the limit as $r \to +\infty$ we have that
\begin{equation}
    (1- \varepsilon) \AVR(g)\left( \frac{n-p}{p-1} \right)^{p-1} \leq\lim_{r\to +\infty} \frac{ \capa_p(\set{\rho \leq r})}{r^{n-p} \abs{\S^{n-1}}} \leq (1+\varepsilon)  \AVR(g)\left( \frac{n-p}{p-1} \right)^{p-1}
\end{equation}
which in turns gives \eqref{eq:limit_p_cap_cross_sections} by arbitrariness of $\varepsilon$. \endproof

\subsection{Li-Yau-type estimates and gradient bound}
We will provide Li-Yau-type estimates for the $p$-Green function $G_p$ with a controlled constant as $p\to 1^+$. These estimates will be the starting point for the proof of both \cref{thm:asymptoticbehaviourofppotentialC0,thm:asymptotic_behaviour_of_IMCF}. We highlight that in \cite{Mari2019} the authors provided Li-Yau-type estimates for $p$-harmonic functions. The upper bound in \cite[Theorem 3.6]{Mari2019} carried out in a broader setting is actually in term of the distance. Conversely, the lower bound in \cite[Corollary 2.8]{Mari2019} is in terms of the distance in a model which has the same radial sectional curvature of the lower bound in \cref{eq:Ricci-intro}. 
This estimate does not seem sufficient for our aims. 

However, since our setting disposes of a precise asymptotic structure we can improve such lower bound by inheriting some techniques coming from \cite{Holopainen1999}, and using the natural foliaton of ends induced by the cross-sections of the asymptotic cone. To accomplish this program we first need a global Harnack Inequality holding on each end of $(M,g)$.

\begin{proposition}[Harnack's Inequality]\label{prop:harnack}
Let $(M,g)$ be a $\CC^0$-Asymptotically Conical Riemannian manifold with Ricci curvature satisfying \eqref{eq:Ricci-intro} and let $p>1$. A uniform Harnack Inequality holds on every $\set{\rho =R}\cap E$, that is there exists a constant $\Cn_H>0$ that depends only on the dimension and $p$ but not on $R$ such that
\begin{equation}\label{eq:harnackend}
    \inf_{\set{\rho =R}\cap E} u \geq \Cn_H^{\frac{1}{p-1}} \sup_{\set{\rho =R}\cap E} u
\end{equation}
for every positive $p$-harmonic function $u$ on $E$. Moreover, $\Cn_H$ is bounded as $p\to 1^+$
\end{proposition}

\proof \cref{prop:Sobolev_in_AC} and \cite[Example 2.20]{Holopainen1999} guarantees that the hypotheses of \cite[Theorem 3.4]{Mari2019} are satisfied on $B(x,6r)$ for every $x\in M\smallsetminus B(o,2R)$. Hence we have that
\begin{equation}\label{eq:Harnack-balls}
    \sup_{B(x,R)} u \leq \Cn^{\frac{1}{p-1}} \inf_{B(x,R)} u
\end{equation}
for every positive $p$-harmonic function $u$ on $E$, for some constant which is bounded for large $R$ and in $p$ as $p\to 1^+$. Since the $\set{\rho =R}\cap E$ is connected and its diameter increases linearly in $R$ by the asymptotic assumption, it can be covered with $N$ balls of radius $R$, for some $N$ not depending on $R$. By chaining \cref{eq:Harnack-balls} $N$ times we obtain \cref{eq:harnackend}. \endproof

As a consequence we can overcome the main issue in \cite[Proposition 5.9]{Holopainen1999}, that is a control on the bounded components of $E \smallsetminus B(o,r)$. The proof of \cite[Proposition 5.9]{Holopainen1999} suggest also an explicit value for the constant.

\begin{proposition}\label{prop:Holopainen}
Let $(M,g)$ be a $\CC^0$-Asymptotically Conical Riemannian manifold with Ricci curvature satisfying \eqref{eq:Ricci-intro} and $o \in M$. For every end $E$ of $M$ there exists a constant $\Cn>0$ such that
\begin{align}
    \sup_{\partial E(r)} G(o,x)& \geq \Cn^{\frac{1}{p-1}} \int\limits_{2r}^{+\infty} \left( \frac{t}{\abs{B(o,t) \cap E}} \right)^{\frac{1}{p-1}} \dd t \label{eq:lower_bound_Green}
\end{align}
holds for every $r>0$ where $E(r)$ is the unbounded component of $E\smallsetminus \overline{B(o,r)}$. Moreover, the constant $\Cn= a\, \capa_p(\overline{K};E)$ where $a>0$ is such that $G_p(x,o)\geq a^{1/(p-1)}$ on $M\smallsetminus K$ and $K$ is the compact set in \cref{def:C0ACintro}.
\end{proposition}

\proof  We already observed that $(M,g)$ as a finite number of ends with connected boundary in \cref{lem:finiteends}. Denote by $E$ one of them. Since $E$ is $\CC^0$-Asymptotically Conical, the volume of $B(o,r) \cap E$ grows like $r^n$, then for some large $R>0$
\begin{equation}
    \int\limits_{R}^{+\infty} \left(\frac{r}{\abs{B(o,r) \cap E}} \right)^{\frac{1}{p-1}} \dd r \leq \Cn R^{-\frac{n-p}{p-1}} < +\infty
\end{equation}
that is $E$ is $p$-large. $E$ satisfies a weak $(1,p)$-Poincaré Inequality and a volume-doubling property (see \cite[Example 2.20]{Holopainen1999}. Moreover, $E$ satisfies the volume comparison condition which means that there exists a constant $\Cn_v$ such that
\begin{equation}\label{eq:PVC}
    \abs{E\cap B(o,r)}\leq \Cn_v \abs{B(x,r/8)}
\end{equation}
holds for any $r \geq R$ and $x \in \partial B(o,r) \cap E$. Observe that \eqref{eq:PVC} holds on the cones. Indeed, since $L$ is compact $\abs{B(x,r/8)}\geq \mu r^n/8^n$ for some positive $\mu>0$. If $E$ is merely $\CC^0$-Asymptotically Conical, both the volume of $E \cap B(o,r)$ and $B(x,r/8)$ are approaching the corresponding ones on the cone, which proves \eqref{eq:PVC}. 

With the very same arguments as in \cite[Proposition 5.9]{Holopainen1999} one can prove \eqref{eq:lower_bound_Green}, by restricting all quantities to the given end. A close look to \cite[Proposition 5.9]{Holopainen1999} gives also that the constant can be chosen as above, firstly choosing $a>0$ and following the computations accordingly. \endproof

\endproof


The above Proposition implies the following double-sided bound on $G_p$, that will in turn imply an analogous one for $u_p$. Such double bound will be key for working out the asymptotics.
\begin{theorem}[Li-Yau-type estimates]\label{thm:li-yau}
    Let $(M,g)$ be a $\CC^0$-Asymptotically Conical Riemannian manifold with $\Ric$ satisfying \eqref{eq:Ricci-intro}. Let $p>1$, each end of $M$ is $p$-nonparabolic and, given $G_p$ the $p$-Green function, then there exist two constants $\Cn_L,\Cn_U>0$
    \begin{equation}\label{eq:bound_pgreen}
        \Cn_L^{\frac{1}{p-1}} \d(o,x)^{-\frac{n-p}{p-1}} \leq G_p(o,x) \leq \Cn_U^{\frac{1}{p-1}} \d(o,x)^{-\frac{n-p}{p-1}}.
    \end{equation}
    for every $x \in M\smallsetminus\{o\}$. Moreover, the constant $\Cn_U$ is bounded as $p \to 1^+$ and there exists $R>0$ not depending on $p$ such that the lower bound in \cref{eq:bound_pgreen} holds with constant $\Cn_L= \Cn\, a \,\capa_p( \overline{K};E)$ on $M\smallsetminus B(o,R)$, where $a>0$ is such that $G_p(x,o)\geq a^{1/(p-1)}$ on $M\smallsetminus K$, $\Cn$ does not depend on $p$ and $K$ is the bounded set in \cref{def:C0ACintro}.
\end{theorem}

\proof The upper bound in \cref{eq:bound_pgreen} with $\Cn_U$ bounded as $p \to 1^+$ follows from \cite[Theorem 3.6]{Mari2019}. Indeed, the assumptions are satisfied since a  $p$-Sobolev Inequality holds true as a standard consequence of the $L^1$-Sobolev inequality in \cref{prop:Sobolev_in_AC}.
For what it concerns the lower bound, we are in position to apply \cref{prop:Holopainen}. The main issue is that we do not have control on the bounded components of $E \smallsetminus \overline{B(o,R)}$. Consider the function $R:[1,+\infty) \to \R$ defined as 
\begin{equation}
    R(t) = \max \set{ \d(o,x) \st x \in \set{\rho =t}}.
\end{equation}
Observe that $\set{\rho \geq t} \supset E(2R(t))$. Then applying the Harnack's Inequality \eqref{eq:harnackend}, Comparison Principle and \eqref{eq:lower_bound_Green} one gets that
\begin{align}
    \inf_{\set{\rho=t}} G_p(o,x) &\geq \Cn_H^{\frac{1}{p-1}}\sup_{\set{\rho =t}} G_p(o,x) \geq \Cn_H^{\frac{1}{p-1}} \sup_{\partial E(2R(t))} G_p(o,x)\\&\geq \Cn_4^{\frac{1}{p-1}} \int\limits_{4 R(t)}^{+\infty}\left( \frac{r}{\abs{B(o,r)\cap E}}\right)^{\frac{1}{p-1}} \dd r\geq \Cn_5^{\frac{1}{p-1}}  R(t)^{-\frac{n-p}{p-1}}.
\end{align}
for every $t \geq T$, for some $T$ large enough not depending on $p$ and $\Cn_5=\Cn a \capa_p(\partial K;E)$.  By \cref{lem:asymptoticaldistance} there exists $R_1 \geq R(T)$ such that $\partial B(o,r) \subset \set{\rho \leq 2r}$ and $R(2r) \leq 4r$ hold for every $r\geq R_1$. Then, by the Maximum Principle, 
\begin{equation}
    \inf_{\partial B(o,r)} G_p(o,x) \geq  \inf_{\set{\rho = 2r}} G_p(o,x) \geq \Cn_5^{\frac{1}{p-1}} R(2r)^{-\frac{n-p}{p-1}}\geq \left(\frac{\Cn_5}{4^{n-p}}\right)^{\frac{1}{p-1}} r^{-\frac{n-p}{p-1}},
\end{equation}
holds for every $r\geq 1$, since $R_1$ does not depend on $p$. The global lower bound follows since it is satisfied near the pole $o$, but the new constant $\Cn_L$ might go to $0$ as $p\to 1^+$ \endproof

Arguing as done in the proof of \cite[Theorem 2.15]{Benatti2021}, it is easy to derive analogous estimates for $u_p$. We do not take trace of the best constant as in the previous result, since this tool will be only used to derive \cref{thm:asymptoticbehaviourofppotentialC0} where $1<p<n$ is fixed. 

\begin{corollary}[Li-Yau-type estimates for the $p$-capacitary potential]\label{thm:Li-Yau-estimates}
 Let $(M,g)$ be a $\CC^0$-Asymptotically Conical Riemannian manifold with $\Ric$ satisfying \eqref{eq:Ricci-intro}. Let be $1<p<n$. There exists a unique solution $u_p$ to \eqref{pb-intro} and there exists a positive and finite constant $\Cn$ such that
 \begin{equation}\label{eq:LiYau}
     \Cn^{-1} \d(x,o)^{-\frac{n-p}{p-1}} \leq u_p(x) \leq \Cn \d(x,o)^{-\frac{n-p}{p-1}}
 \end{equation}
 for every $x \in M \smallsetminus \Omega$.
\end{corollary}

\proof The existence follows from \eqref{eq:bound_pgreen} and \cref{thm:existence_of_p_potential}. Moreover, In light of \eqref{eq:bound_pgreen}, it suffices to show that there exists a positive finite constant $\Cn$ such that $\Cn^{-1} G_p \leq u_p \leq \Cn G_p$. Choose any $ \sup_{\partial \Omega} u_p< \Cn$. Then, $\Cn^{-1} G_p < u_p$ on $\partial \Omega$. Moreover, since both $u_p$ and $G_p$ vanish at infinity, for any $\delta > 0$ we have $\Cn^{-1} G_p < u_p + \delta$ on $\partial B(o, R)$ for any $R$ big enough. The Comparison Principle applied to the $p$-harmonic functions $u_p + \delta$ and $G_p$ in $B(o, R) \smallsetminus \overline{\Omega}$ shows that $\Cn^{-1} G_p < u_p + \delta$ in the latter subset. The radius $R$ being arbitrarily big, this implies that, by passing to the limit as $R \to + \infty$, that $\Cn^{-1} G_p < u_p + \delta$  in the whole $M \smallsetminus \Omega$. Letting $\delta \to 0^+$ leaves with $\Cn^{-1} G_p \leq u_p$, and consequently with the lower bound in \eqref{eq:LiYau}. The inequality $u_p \leq \Cn G_p$, yielding the upper bound, is shown the same way. \endproof

We recall the following Cheng-Yau-type estimate, proved in \cite{Wang2010}, together with the consequent gradient bound for the $p$-capacitary potential. 

\begin{theorem}[Cheng-Yau-type estimate]\label{thm:p-cheng-yau}
Let $(M,g)$ a complete Riemannian manifold. Let $v\in W^{1,p}_{\loc}B(o,2R)$ be a positive $p$-harmonic function on a geodesic ball $B(o,2R)$ for some $R>0$ where $\Ric \geq -(n-1)\kappa^2$. Then there exists a constant $\Cn=\Cn(p,n)$ such that
\begin{equation}\label{eq:p-cheng-yau}
    \sup_{B(o,R)} \abs{ \D \log(v)} \leq \Cn\left( \frac{1}{R}+ \kappa\right).
\end{equation}
\end{theorem}
The above local estimate enables us to provide a global gradient bound for $u_p$, that will be key for its asymptotic analysis.
\begin{proposition}\label{prop:gradient-bound-p-armoniche}
Let $(M,g)$ be a complete $\CC^0$-Asymptotically Conical manifold with $\Ric$ satisfying \eqref{eq:Ricci-intro}. Let $\Omega \subset M$ open bounded with smooth boundary and let $u_p$ be its $p$-capacitary potential. Then, there exists a constant $\Cn>0$ such that
\begin{equation}\label{eq:Asymptotically-gradientbound}
    \abs{\D \log u_p }\leq \frac{\Cn}{\dist(x,o)}
\end{equation}
holds on the whole $M \smallsetminus \Omega$.
\end{proposition}

\begin{proof}
By the $\mathscr{C}^{1}$-regularity of $u$, it clearly suffices to show that \eqref{eq:Asymptotically-gradientbound} holds true outside some compact set containing ${\Omega}$. Let then $o \in \Omega$ and $R > 0$ be such that $\Omega \subset B(o, R)$, and let $x \in M \smallsetminus \overline{B(o, 2R)}$. With this choice, we have $B(x, \dist(o, x) - R) \subset M \smallsetminus \overline{B(o, R)}$. Thus, applying inequality \eqref{eq:p-cheng-yau} to the function $u_p$ in the ball $B(x, \dist(o, x) - R)$ we get
\begin{align}
    \frac{\abs{\D u_p}}{u_p}&\leq 2 \Cn \left(\frac{1}{\dist(o,x)-R}+\frac{\kappa}{\dist(o,x)+1}\right) \leq 2\Cn\frac{(2+ \kappa)}{\dist(o,x)}.
\end{align}
concluding the proof.
\end{proof}

The constant in \cref{eq:p-cheng-yau} and consequently in \cref{eq:Asymptotically-gradientbound} is such that $(p-1)\Cn$ diverges as $p \to 1^+$. Up to the authors' knowledge a Cheng-Yau-type estimate with $(p-1) \mathrm{C}$ controlled in $p$ is yet to be discovered. This should lead to various other insights about the weak IMCF and its relations with $p$-harmonic potentials. 

\section{Asymptotic behaviour of the \texorpdfstring{$p$}{p}-capacitary potential}
\label{sec:asy-p}

In this section we prove \cref{thm:asymptoticbehaviourofppotentialC0}. In fact, as anticipated in the Introduction, we prove a more general statement that provides information also about the asymptotic behaviour of the derivatives of $u$, if the asymptotic structure of the underlying metric is suitably reinforced.

\begin{theorem}[Asymptotic behaviour of the derivatives of $p$-capacitary potential]\label{thm:asymptoticbehaviourofpotential}
Let $(M, g)$ be a complete $\CC^{k,\alpha}$-Asymptotically Conical Riemannian manifold for some $\alpha >0$ and $k \in \N$ with $\Ric$ satisfying \cref{eq:Ricci-intro}. Let $E_1, \ldots, E_N$ be the (finitely many) ends of $M$ with respect to the compact $K$ in \cref{def:C0ACintro}. Consider $\Omega \subset M$ be an open bounded subset with smooth boundary and $u:M\smallsetminus \Omega \to \R$ a solution of the problem \eqref{pb-intro}. Then
\begin{equation}\label{eq:asymptotic_p_capacitary_higher_order}
    \abs{\D^\ell u - \left( \frac{ \Cn^{(i)}_p( \Omega)}{ \AVR(g; E_i)}\right)^{\frac{1}{p-1}}\D^\ell \rho^{-\frac{n-p}{p-1}}}_{\hat{g}} =o\left( \rho^{-\frac{n-p}{p-1}-\ell}\right) 
\end{equation}
on $E_i$ as $\dist(o,x)\to +\infty$ for every $i=1, \ldots, N$ and $\ell \leq k+1$.
\end{theorem}
For future reference we want to specify the behaviour of partial derivatives coming from \cref{eq:asymptotic_p_capacitary_higher_order}. Given a coordinate system $(\vartheta^1, \ldots, \vartheta^{n-1})$ on $L$ one has that
\begin{equation}
    \frac{\partial^{j+\abs{\alpha}} u }{\partial \rho^j\partial \theta^\alpha} =  \left( \frac{ \Cn^{(i)}_p( \Omega)}{ \AVR(g; E_i)}\right)^{\frac{1}{p-1}}  \frac{\partial^{j+\abs{\alpha}} }{\partial \rho^j\partial \theta^\alpha} \left( \rho^{-\frac{n-p}{p-1}}\right)+ o \left( \rho^{-\frac{n-p}{p-1}- j}\right)
\end{equation}
as $\rho \to +\infty$, where $\alpha$ is a $(n-1)$-dimensional multi-index such that $j+\abs{\alpha}\leq k$.


Along with the proof, we extend \cref{lem:limit_p_cap_cross_sections}, showing that the $p$-capacity of the $p$-capacitary potential behaves like the $p$-capacity of the cross sections approaching infinity.

\begin{proposition}[Asymptotic behaviour of the $p$-capacity of level sets]
\label{prop:p-capacity_level_sets}
In the same assumptions and notations of \cref{thm:asymptoticbehaviourofppotentialC0}, set, for $i = 1, \dots, N$, 
\begin{equation}
v_i = \left(\frac{\Cn^{(i)}_p(\Omega)}{\AVR(g; E_i)}\right)^{\frac{1}{n-p}} u^{-\frac{p-1}{n-p}}.
\end{equation}
Then, we have
\begin{equation}
\lim_{s \to + \infty} \frac{\capa_p(\set{v_i\leq s} \cap E_i;E_i)}{s^{n-p} \abs{\mathbb{S}^{n-p}}} = \left(\frac{n-p}{p-1}\right)^{p-1} \AVR(g; E_i).
\end{equation}
\end{proposition}

Moreover, as a byproduct, we obtain the following uniqueness result on Riemannian cones.

\begin{proposition}\label{prop:uniqueness_on_cones}
Let $((0,+\infty) \times L , \hat{g})$ be a  Riemannian cone with nonnegative Ricci curvature where $L$ is a closed connected smooth hypersurface. Let $u$ be a nonnegative $p$-harmonic function on $(0,+\infty) \times L$ satisfying $u(x) \leq \Cn \rho(x)^{-(n-p)/(p-1)}$ for every $x \in (0,+\infty) \times L$ for some constant $\Cn\geq 0$. Then, there exists a nonnegative $\gamma \in \R$ such that
\begin{equation}
    u(x)= \gamma \rho(x)^{-\frac{n-p}{p-1}}
\end{equation}
holds on $(0,+\infty)\times L$.
\end{proposition}

\proof[Proof of \cref{thm:asymptoticbehaviourofppotentialC0,thm:asymptoticbehaviourofpotential,prop:uniqueness_on_cones,prop:p-capacity_level_sets}] It is enough to prove the theorems in the case $M$ has only one end. The proof of the general case then follows applying the result to each end. We will denote by $g^{}_{(s)}$ the metric $s^{-2} \omega_s^* g$ on $[1/s,+\infty)\times L$, being $\omega_s$ the family of diffeomorphisms defined in \cref{eq:dilatation}. We find convenient to organise the proof in four steps. The first three steps are devoted do prove \cref{thm:asymptoticbehaviourofppotentialC0}. The second and the third one contain the proof of \cref{prop:uniqueness_on_cones} and \cref{prop:p-capacity_level_sets} respectively. In the last we complete the proof for higher order asymptotic behaviour mainly using Schauder estimates for $p$-harmonic functions.

\step\label{step:stepAsymptoticTheorem}
Suppose that $(M,g)$ is $\CC^{0}$-Asymptotically Conical. Define the family of functions $u_s:[1/s,+\infty) \times L\to \R$ as
\begin{equation}
    u_s(x)= s^{\frac{n-p}{p-1}} u\circ\omega_s(x)
\end{equation}
where $\omega_s$ is the map defined in \cref{eq:dilatation}. The aim of this step is to prove compactness of $(u_s)_{s\geq 1}$ with respect to local uniform convergence on $(0,+\infty) \times L$. In particular, by \cref{thm:compactness_uniform_convergence_p_harmonic} (see also \cref{rmk:compactness_convergence_of_metric}) any limit point $w$ for the sequence $(u_s)_{s \geq 1}$ is $p$-harmonic with respect to the metric $\hat{g}$ on $(0,+\infty) \times L$. Moreover, there exists a positive constant $\Cn$ such that
\begin{equation}\label{eq:Li-Yau-P-armonic-limit}
    \Cn^{-1} \rho(x)^{-\frac{n-p}{p-1}} \leq w(x) \leq \Cn \rho(x)^{-\frac{n-p}{p-1}}
\end{equation}
is satisfied for every $x \in (0,+\infty) \times L$.

By \cref{thm:Li-Yau-estimates} we have that
	\begin{equation}\label{eq:holopainenconsequence}
	    \Cn_1^{-1} \left(\vphantom{\sum}\dist(o,x)\right)^{-\frac{n-p}{p-1}} \leq u(x) \leq  \Cn_1 \left(\vphantom{\sum}\dist(o,x)\right)^{-\frac{n-p}{p-1}}.
	\end{equation}
holds on $M \smallsetminus \Omega$. In particular, since by \cref{lem:asymptoticaldistance} the distance function from $o$ behaves asymptotically as the coordinate $\rho$, we deduce that there exist $S_2,\Cn_2>0$ such that 
    \begin{equation}\label{eq:li-yau-u_s}
      \Cn_2^{-1} \rho(x)^{-\frac{n-p}{p-1}} \leq u_s(x) \leq \Cn_2 \rho(x)^{-\frac{n-p}{p-1}},
    \end{equation}
holds on $[1/s,+\infty)\times L$ for every $s\geq S_2$. In particular, $(u_s)_{s\geq 1}$ is equibounded.
By the gradient estimate \cref{prop:gradient-bound-p-armoniche}
\begin{equation}
    \abs{ \D u}(x) \leq \Cn_3 u(x)^{\frac{n-1}{n-p}} \leq \Cn_4 \left(\vphantom{\sum} \dist(o,x) \right)^{-\frac{n-1}{p-1}}
\end{equation}
for some positive constants $\Cn_3,\Cn_4$. Hence, employing again \cref{lem:asymptoticaldistance} there exist $S_5, \Cn_5>0$ such that
\begin{equation}\label{eq:bound_on_gradient_hatg}
\abs{\D u_s}_{g_{(s)}}(x) \leq \Cn_5 \rho(x)^{-\frac{n-1}{p-1}}
\end{equation}
holds on $[1/s,+\infty)\times L$ for every $s \geq S_5$. By \cref{lem:convergence} we have that for some $\varepsilon>0$ there is $S_6>0$ such that
\begin{equation}
    \abs{\D u_s}_{\hat{g}} \leq (1+ \varepsilon) \abs{ \D u_{s}}_{g_{(s)}}
\end{equation}
holds for every $s \geq S_6$. Combining it with \cref{eq:bound_on_gradient_hatg} we obtain that the family $(u_s)_{s \geq 1}$ is equicontinuous. By Arzelà-Ascoli Theorem we conclude that $(u_s)_{s \geq 1}$ is precompact with respect to the local uniform convergence. \cref{eq:Li-Yau-P-armonic-limit} follows from \cref{eq:li-yau-u_s}.

\step\label{step:gammadefconvergence} Here we prove that any limit point $v$ of the family $(u_s)_{s \geq 1}$ has the form
\begin{equation}\label{eq:uniqueness_proof}
    v(x)= \gamma \rho (x)^{-\frac{n-p}{p-1}},
\end{equation}
for some nonnegative $\gamma \in \R$, proving also \cref{prop:uniqueness_on_cones}.
Let $v:(0,+\infty) \times L \to \R $ be a nonnegative $p$-harmonic function satisfying the bound $v(x) \leq \Cn \rho(x)^{-\frac{n-p}{p-1}}$ on $(0,+\infty) \times L$. 

Define the function $\ecc_v:(0,+\infty)\to \R$ as
\begin{equation}\label{eq:eccentricity}
    \ecc_v(t) = \frac{R(t)}{r(t)},
\end{equation}
where $[r(t),R(t)] \times L$ is the smallest annulus containing $\set{v=1/t}$ for every $t \in (0,+\infty)$.
Observe that, $\ecc_v(t)\geq 1$ and $\ecc_v(t)=1$ if and only if $\set{v=1/t}$ is a cross-section of the cone. By the Comparison Principle \cref{thm:comparison_principle}, using the potentials of $\set{\rho=r(t)}$ and $\set{\rho =R(t)}$ as barriers, we have that
\begin{equation}\label{eq:comparison-control}
    r(t)\left( \frac{t}{T}\right)^{\frac{n-p}{p-1}} \leq \rho(x) \leq R(t)\left( \frac{t}{T}\right)^{\frac{n-p}{p-1}} 
\end{equation}
holds for every $x \in \set{v=1/T}$ for every $T \geq t$. Hence, $\ecc_v$ is nonincreasing. Moreover, since $(0,+\infty) \times L$ is connected, $\rho(x)^{-\frac{n-p}{p-1}}$ is $p$-harmonic and $\CC^2((0,+\infty) \times L)$ and $\abs{ \D \rho^{-\frac{n-p}{p-1}}} \geq \frac{n-p}{p-1} R^{-\frac{n-1}{p-1}}$ holds on $(0,R) \times L$ for every $R>0$, by the Strong Comparison Principle \cref{thm:comparison_principle} the inequalities in \cref{eq:comparison-control} are strict unless $\set{v=t}$ is a cross-section. It is not hard to see that $\ecc_v$ is scale invariant

Consider for $s \leq 1$ the family $v_s:[1,+\infty) \times L\to \R$ defined as
\begin{equation}\label{eq:blowup-sequence}
    v_s(x)= s^{\frac{n-p}{p-1}} v \circ \omega_s(x)
\end{equation}
where $\omega_s$ is defined in \eqref{eq:dilatation}. 
 Using the same argument of \cref{step:stepAsymptoticTheorem} we have that
\begin{align}
    v_s(x) \leq \Cn \rho(x)^{-\frac{n-p}{p-1}} &&\text{and}&& \abs{ \D v_s}(x) \leq \Cn \rho(x)^{-\frac{n-1}{p-1}}
\end{align}
holds on $(0,+\infty) \times L$ for some constant $\Cn>0$. Hence, appealing to the Arzelà-Ascoli Theorem, $(v_s)_{s\leq 1}$ is precompact with respect to the local uniform convergence. Let $w$ be a limit point for $(v_s)_{s \leq 1}$. Since $\ecc$ is scale invariant, $\ecc_{v_s}(t)= \ecc_{v}(t/s)$. Then, $\ecc_w(t)$ is constant equal to some $\ecc_w \in [1,+\infty)$ that by monotonicity satisfies $\ecc_w= \sup_t \ecc_v(t) \in [1,+\infty)$. Suppose by contradiction that $\ecc_w>1$. Then the level $\set{w=1} \subset [r(1), \ecc_w r(1)] \times L$ and $\set{w=1}$ touches both the cross-sections $\set{\rho=r(1)}$ and $\set{\rho =\ecc_wr(1)}$ without being equal to either one. By \cref{eq:comparison-control} and the Strong Comparison Principle \cref{thm:comparison_principle}
\begin{equation}
    r(1) t^{\frac{p-1}{n-p}} < \rho(x) < \ecc_w r(1) t^{\frac{p-1}{n-p}}
\end{equation}
holds for every $x \in \set{w=1/t}$ for every $t>1$. We therefore have that $\ecc_w(t) < \ecc_w$ which is a contradiction. In conclusion $\ecc_w$ must be $1$ and since $ 1 \leq \ecc_v(t) \leq \ecc_w=1$, $v$ is as in \eqref{eq:uniqueness_proof}.

\step By \cref{step:gammadefconvergence}, any limit point $w$ for the sequence $(u_{s})_{s \geq 1}$ given by \cref{step:stepAsymptoticTheorem} has the form $\gamma \rho^{-\frac{n-p}{p-1}}$, where $\gamma >0$ by \eqref{eq:Li-Yau-P-armonic-limit}. We are now going to prove that
\begin{equation}\label{eq:gamma-definition}
    \gamma= \Cn_p(\Omega)^{\frac{1}{p-1}}\AVR(g)^{-\frac{1}{p-1}}.
\end{equation}
The characterisation \eqref{eq:gamma-definition} ensures that any converging subsequence admits the same limit, proving that the whole family $(u_s)_{s\geq 1}$ locally uniformly converges to $\gamma \rho^{-\frac{n-p}{p-1}}$ as $s \to +\infty$. In particular, for every $\varepsilon>0$ there exists a $S\geq 1$ such that
\begin{equation}
\label{u-gamma}
\begin{split}
     \sup_{\set{\rho =s}}s^{\frac{n-p}{p-1}}\abs{u -\left(\frac{\Cn_p(\Omega)}{\AVR(g)}\right)^{\frac{1}{p-1}}\rho^{-\frac{n-p}{p-1}}}&= \sup_{\set{\rho=1}}\abs{u_s -\left(\frac{\Cn_p(\Omega)}{\AVR(g)}\right)^{\frac{1}{p-1}}\rho^{-\frac{n-p}{p-1}}}\leq \varepsilon
\end{split}
\end{equation}
for every $s\geq S$, proving \cref{thm:asymptoticbehaviourofppotentialC0,prop:p-capacity_level_sets}.

Observe that $\gamma>0$ by \cref{thm:Li-Yau-estimates}. In order to prove \eqref{eq:gamma-definition}, we find convenient to work with the auxiliary function
\begin{equation}\label{eq:auxiliaryfunction}
v= \left( \frac{u}{\gamma}\right)^{-\frac{p-1}{n-p}}.
\end{equation}
Since $w$ is a limit point for the family $(u_s)_{s \geq 1}$ there is a subsequence $(u_{s_k})_{k \in \N}$, $s_k$ increasing and divergent as $k\to +\infty$, such that $u_{s_k}\to v = \gamma \rho^{-\frac{n-p}{p-1}}$ locally uniformly on $(0,+\infty) \times L$ as $k \to +\infty$. Then for any $\varepsilon>0$ there exists $k_\varepsilon \in \N$ such that
\begin{equation}
    \set{ \rho \leq\frac{s_k}{1+\varepsilon}} \subset \set{ v \leq s_k} \subset \set{\rho \leq\frac{s_k}{1-\varepsilon}}
\end{equation}
By the monotonicity of the $p$-capacity with respect to the inclusion, we have that
\begin{equation}
    \capa_p\left(\set{ \rho \leq\frac{s_k}{1+\varepsilon}}\right)\leq \capa_p\left(\set{v \leq s_k}\right) \leq \capa_p\left(\set{\rho \leq\frac{s_k}{1-\varepsilon}}\right)
\end{equation}
By \cref{eq:scaling-invariant-cap} we can compute the capacity of level sets of $v$ in terms of the capacity of $\partial \Omega$, that is
\begin{equation}
    \capa_p\left(\set{ \rho \leq\frac{s_k}{1+\varepsilon}}\right)\leq\gamma^{-(p-1)} s_k^{n-p} \capa_p\left( \Omega\right) \leq \capa_p\left(\set{\rho \leq\frac{s_k}{1-\varepsilon}}\right)
\end{equation}
Dividing each side by $\abs{\S^{n-1}}s_k^{n-p}$, sending $k \to +\infty$ and using \cref{lem:limit_p_cap_cross_sections} we infer that
\begin{equation}
    \left(\frac{n-p}{p-1}\right)^{p-1}\frac{ \AVR(g)}{(1+\varepsilon)^{n-p}}\leq \gamma^{-(p-1)} \frac{\capa_p( \Omega) }{\abs{\S^{n-1}}}\leq \left(\frac{n-p}{p-1}\right)^{p-1}\frac{\AVR(g)}{(1-\varepsilon)^{n-p}}
\end{equation}
Then \cref{eq:gamma-definition} follows by arbitrariness of $\varepsilon>0$, keeping in mind the characterisation of $\AVR(g)$ in \eqref{def:AVRconico} and the relation between the $p$-capacity and the normalised $p$-capacity.

\step Suppose now $(M,g)$ is $\CC^{0,\alpha}$-Asymptotically Conical for $\alpha>0$. By \cref{thm:regularityp-potential} $u_s \in \CC_{\loc}^{1,\beta}((1/s,+\infty) \times L)$ for some $\beta \in (0, \alpha)$ and for every $K \subset (1/s,+\infty) \times L$ there exists  constant $\Cn>0$ such that
\begin{align}\label{eq:first_order_regularity_asymptotic_proof}
\norm{u_s}_{\CC^{1,\beta}(K)}\leq \Cn\norm{u_s}_{\CC^0((1/s,+\infty) \times L)}.
\end{align}
Since the metric $g_{(s)}$ locally $\CC^{0,\alpha}$ -converges to $\hat{g}$ on $(0,+\infty) \times L$ by \cref{lem:convergence}, the constant in \cref{eq:first_order_regularity_asymptotic_proof} can be chosen not depending on $s$. Hence, by Arzelà-Ascoli Theorem, $(u_s)_{s \geq 1}$ $\CC^{1}$-locally converges on $(0,+\infty) \times L$ to the function
\begin{equation}\label{eq:limit_function}
    \left( \frac{ \Cn_p( \Omega)}{\AVR(g)}\right)^{\frac{1}{p-1}} \rho ^{-\frac{n-p}{p-1}}
\end{equation}
as $s \to +\infty$. This proves \cref{thm:asymptoticbehaviourofpotential} for $k=0$ and $\ell\leq 1$.

If $(M,g)$ is $\CC^{k,\alpha}$-Asymptotically Conical for $k\geq 1$ and $\alpha >0$, we already proved that \eqref{eq:asymptotic_p_capacitary_higher_order} holds for $\ell \leq 1$. In particular, for every $ R$ there exists $S>0$ such that $\abs{\D u_s}> 0$ holds on every compact $K\subset (R,+\infty) \times L$ for every $s\geq S$. Applying \cref{thm:moreregularityp-potential}, $u_s \in \CC^{\infty}((R,+\infty) \times L)$ for every $s \geq S$. Moreover, for every $K \subset (R,+\infty) \times L)$ there exists a constant $\Cn>0$ such that
\begin{align}\label{eq:higher_order_regularity_asymptotic_proof}
\norm{u_s}_{\CC^{k+1,\alpha}(K)}\leq \Cn \norm{u_s}_{\CC^0((1/s,+\infty) \times L)}.
\end{align}
Since $g^{}_{(s)}$ locally $\CC^{k+1,\alpha}$-converges to $\hat{g}$ on $(0,+\infty) \times L$ , the constant in \cref{eq:higher_order_regularity_asymptotic_proof} can be chosen not depending on $s$. Since $R$ is arbitrary, $( u_s)_{s\geq 1}$ is precompact with respect to the local $\CC^{k+1}$-topology. Hence, $(u_s)_{s\geq 1}$ converges on compact subsets of $(0,+\infty) \times L$ up to its $(k+1)$-th derivative to the function defined in \cref{eq:limit_function} as $s \to +\infty$, concluding the proof of \cref{thm:asymptoticbehaviourofpotential}. \endproof

As a consequence of \cref{thm:asymptoticbehaviourofppotentialC0}, we extend \cref{lem:AVRconico} showing that the volume of level sets of a suitable function of the $p$-capacitary potential behaves like the volume of geodesic balls approaching infinity. Requiring the assumption of \cref{thm:asymptoticbehaviourofpotential} we actually deduce that the same occur for the $(n-1)$-Hausdorff measure of the level sets.

\begin{proposition}[Asymptotic behaviour of the area of level sets]\label{prop:arelevel}
In the same assumptions and notations of \cref{thm:asymptoticbehaviourofppotentialC0}, set, for $i = 1, \dots, N$, 
\begin{equation}
v_i = \left(\frac{\Cn^{(i)}_p(\Omega)}{\AVR(g; E_i)}\right)^{\frac{1}{n-p}} u^{-\frac{p-1}{n-p}}.
\end{equation}
Then, we have
\begin{equation}\label{eq:volume_level_sets}
\AVR(g; E_i)= \lim_{s\to +\infty} \frac{ \abs{\set{ v_i \leq s}\cap E_i}}{s^n \abs{\B^n}}.
\end{equation}
Moreover, if in addition $(M, g)$ is $\mathscr{C}^{0, \alpha}$-Asymptotically Conical for $\alpha \in (0, 1)$, then
\begin{equation}\label{eq:area_level_sets}
   \AVR(g; E_i)= \lim_{s \to + \infty} \frac{\abs{\{v_i = s\} \cap E_i}}{s^{n-1} \abs{\mathbb{S}^{n-1}}}.
\end{equation}
\end{proposition}

\proof[Proof of \cref{prop:arelevel}]
We prove the statement in the case $M$ has only one end, being the general case a direct consequence. We drop the subscript $i$ in the following lines. By \cref{thm:asymptoticbehaviourofppotentialC0}, for any $\varepsilon>0$ there exists $R_\varepsilon>0$ such that
\begin{equation}
    (1-\varepsilon)\rho \leq v \leq  (1+\varepsilon)\rho
\end{equation}
holds on $\set{\rho \geq R_\varepsilon}$. Thus we have that
\begin{equation}
    \set{\rho \leq \frac{s}{1+\varepsilon}} \subset \set{v\leq s}\subset \set{\rho \leq \frac{s}{1-\varepsilon}} .
\end{equation}
By the monotonicty of the volume we have that
\begin{equation}
    \abs{ \set{\rho \leq \frac{s}{1+\varepsilon}}} \leq \abs{\set{v\leq s}}\leq \abs{ \set{\rho \leq \frac{s}{1-\varepsilon}}}. 
\end{equation}
dividing each side by $s^{n}\abs{\mathbb{B}^n}$ and passing to the limit as $s\to +\infty$ we can conclude that
\begin{equation}
   \frac{\AVR(g)}{(1+\varepsilon)^n} \leq \lim_{s \to +\infty}\frac{\abs{\set{v\leq s}}}{s^n\abs{\B^n}} \leq \frac{\AVR(g)}{(1-\varepsilon)^n}
\end{equation}
by arbitrariness of $\varepsilon>0$ we have that the volume of level sets behaves like the one of geodesic balls approaching infinity, proving the first identity in \cref{eq:volume_level_sets}. A straightforward computation relying on the identity
\begin{equation}
    \abs{\D v} = \left(\frac{\Cn_p( \Omega)}{\AVR(g)}\right)^{\frac{1}{n-p}} \left( \frac{p-1}{n-p}\right) u^{-\frac{n-1}{p-1}}\abs{\D u}
\end{equation}
permits to write
\begin{equation}
    \AVR(g) = \frac{1}{\abs{\S^{n-1}} s^{n-1}} \int\limits_{\set{v=s}} \abs{\D v}^{p-1} \dd \sigma.
\end{equation}
If in addition $(M, g)$ is $\mathscr{C}^{0, \alpha}$-Asymptotically Conical for $\alpha \in (0, 1)$, then $\abs{\D v}$ approaches $1$ at infinity, hence we have 
\begin{equation}
    \AVR(g)= \lim_{s \to +\infty} \frac{1}{\abs{\S^{n-1}} s^{n-1}} \int\limits_{\set{v=s}} \abs{\D v}^{p-1} \dd \sigma = \lim_{s\to +\infty} \frac{ \abs{ \set{ v=s}}}{s^{n-1} \abs{\S^{n-1}}}
\end{equation}
which concludes the proof of \cref{eq:area_level_sets}.
\endproof

\section{Asymptotic behaviour of the weak IMCF}
\label{sec:asy-imcf}
We give the precise definition of (proper) weak Inverse Mean Curvature Flow (IMCF for short).

Given an open set $U \subseteq M$,
we say that $w \in \mathrm{Lip}_{\mathrm{loc}} (U)$ is a weak IMCF if
for every $v \in \mathrm{Lip}_{\mathrm{loc}}(U)$ with $\{w \neq v\} \Subset U$ and any compact set $K\subset U$ containing $\{w \neq v\}$, we have
\begin{equation}
\label{wimcf1}
J_w^K(w) \leq J_w^K(v)
\end{equation}
 where 
\begin{equation}
\label{J}
J_w^K(v) = \int_K \abs{\D v} + v \abs{\D w} \dd\mu. 
\end{equation}
Moreover, given a bounded open set $\Omega \subset M$ with smooth boundary we say that $w$ is a weak IMCF starting at $\Omega$ if
for every $v \in \mathrm{Lip}_{\mathrm{loc}}(M)$ with $\{w \neq v\} \Subset M \setminus {\Omega}$ and any compact set $K\subset M \setminus \Omega$ containing $\{w \neq v\}$ we have
\begin{equation}
\label{wimcf1}
J_w^K(w) \leq J_w^K(v),
\end{equation}
and the set $\Omega$ is the $0$-sublevel set of $w$, that is
\begin{equation}
\label{0-sub}
\Omega = \{w < 0\}.
\end{equation}
We say that $w$ is a proper weak IMCF if the function $w$ is proper, that is, its sublevel sets are precompact.
In the  following result we gather the existence and the fundamental estimates for $w$ in the Asymptotically Conical setting.
\begin{theorem}\label{thm:weak_imcf_AC_to_Riccinonneg}
\label{thm:existenceIMCF-estimates}
Let $(M,g)$ be a $\CC^0$-Asymptotically Conical Riemannian manifold satisfying the Ricci curvature bound \eqref{eq:Ricci-intro}. Let $\Omega \subset M$ be an open bounded subset with smooth boundary. Then, there exists a proper weak IMCF $w$ starting at $\Omega$. Moreover, given $o \in \Omega$, the function $w$ satisfies
\begin{align}\label{eq:li-yau-imcf-C0AC}
    (n-1) \log \dist(x,o) -\Cn\leq w \leq (n-1) \log \dist(x,o)
     +\Cn
\end{align}
where $\Cn= \Cn(M,n, \Omega)$.
\end{theorem}
\proof The existence is guaranteed by \cite[Theorem 1.7]{Mari2019} whose assumptions are satisfied in virtue of \cref{eq:L1Sobolev} and \cref{eq:Ricci-intro}. The lower bound is the consequence of \cite[Theorem 1.7 and 1.3]{Mari2019}. Let $R$ be such that
\begin{equation}\label{eq:lower_bound_IMCF}
    -(p-1) \log G_p(x,o) \leq (n-p) \log \dist(x,o)- \log \Cn_L
\end{equation}
on $M\smallsetminus B(o,R)$ with the constant $\Cn_L= \Cn\, a \,\capa_p( \overline{K};E)$ described in the statement of \cref{thm:li-yau}. By \cite{Mari2019}, $-(p-1) \log G_p(x,o)$ locally uniformly converges as $p\to 1^+$. Then we can choose $a$ in the constant $\Cn_L$ independent from $p$ so that $G_p(o,x)\geq a^{\frac{1}{p-1}}$ holds on $M \smallsetminus K$ where $K$ is as in \cref{def:C0ACintro}. Moreover, by \cref{eq:convergence_p-cap_potential_SOMH}, since $\partial K$ is smooth, $\capa_p(\overline{K};E)$ is bounded as $p\to 1^+$. Hence the constant $\Cn_L$ in \cref{eq:lower_bound_IMCF} does not depend on $p$. Passing to the limit as $p \to 1^+$, in virtue of the upper bound in \cite[Theorem 1.7]{Mari2019}, we obtain
\begin{equation}
    w \leq (n-1) \log \dist(x,o) + \Cn
\end{equation}
outside some $B(o,R)$. Since both the left and side and the right hand side are continuous, the bound can be extended to $M \smallsetminus \Omega$. \endproof

In \cite[Remark 4.9]{Mari2019} the authors also obtained a gradient bound that reads as
\begin{align}
    \abs{ \D w} \leq \frac{\Cn}{\dist(x,o)^{1/\kappa'}}, && \text{where }\kappa'= \frac{1 + \sqrt{1+4\kappa^2}}{2}\geq 1,
\end{align}
for some constant $\Cn>0$ depending only on $\Omega$, the dimension $n$ and the geometry of the ambient manifold. By \cite{Greene1979} (see also \cite[Remark 4.5]{Mari2019}) The exponent $\kappa'$ can be chosen equal to $1$ if the lower bound on the Ricci curvature is of the kind $\Ric \geq -(n-1)f(\dist(x,o))$ for some smooth nonnegative function $f(t)$, such that
\begin{equation}\label{eq:integral_Ricci_condition}
    \int\limits_0^{+\infty} t f(t) \dd t<+\infty.
\end{equation}
The Weak Existence Theorem 3.1 \cite{Huisken2001} the function $w$ satisfies
\begin{equation}\label{eq:HI_Cheng-Yau}
    \abs{ \D w}(x) \leq \sup_{\partial \Omega \cap B(x,r)} \HH + \frac{\Cn}{r}
\end{equation}
for almost every $x \in M \smallsetminus \Omega$ and for every $r$ for which there exists a function $\psi \in \CC^2(B(x,r))$ such that $\psi \geq \dist(x,\, \cdot \,)^2$, $\psi(x)=0$, $\abs{ \D \psi} \leq 3 \dist(x,\, \cdot \,)$, $\D^2 \psi \leq 3 g$ and $\Ric\geq -\Cn/r^2$ in $B(x,r)$. The existence of $\psi$ is guaranteed if a sectional curvature lower bound is ensured. Otherwise, one can require a higher rate of convergence of the metric.
\begin{proposition}\label{prop:Cheng-Yau-Improved}
Let $(M,g)$ be a $\CC^1$-Asymptotically Conical Riemannian manifold satisfying the Ricci curvature bound \eqref{eq:Ricci-intro}. Let $\Omega \subset M$ be an open bounded subset with smooth boundary, $o\in M$. There exists a postive constant $\Cn= \Cn(n,M,\Omega)>0$, such that the solution $w \in \Lip(M)$ of the weak IMCF starting at $\Omega$, given by \cref{thm:weak_imcf_AC_to_Riccinonneg} satisfies
\begin{equation}\label{eq:Cheng-Yau-AC-Ricciqv}
    \abs{ \D w}(x) \leq \frac{ \Cn}{\dist(x,o)}
\end{equation}
for almost every $x \in M \smallsetminus \Omega$.
\end{proposition}

\proof By \cite[Weak Existence Theorem 3.1]{Huisken2001} the function $w$ satisfies \eqref{eq:HI_Cheng-Yau} for almost every $x \in M \smallsetminus \Omega$. In virtue of the discussion in \cite[Definition 3.3]{Huisken2001} (see also the proof of \cite[Blowdown Lemma 7.1]{Huisken2001}) there exists a constant $\Cn>0$ and $R>0$ such that $r \geq \Cn \dist(x,o)$ in \eqref{eq:HI_Cheng-Yau} for every $x\in M\smallsetminus B(o,R)$. Then \eqref{eq:Cheng-Yau-AC-Ricciqv} follows taking $r$ so that $\partial \Omega \cap B(x,r)= \varnothing$. \endproof
The notion of weak IMCF is intimately tied with that of strictly outward minimising sets \cite[Minimizing Hulls, p.371]{Huisken2001}. Leaving the details to \cite{Fogagnolo2020a}, we recall that a bounded set $E \subset M$ is strictly outward minimising if $P(F) > P(E)$ for any bounded $F \supset E$ with $\abs{F \setminus E} > 0$. The following simple lemma implies the existence of a foliation of strictly outward minimising sets in Asymptotically Conical manifolds.

\begin{lemma}\label{lem:cross-sections-SOM}
Let $(M,g)$ be a $\CC^0$-Asymptotically Conical Riemannian manifold. Then $\set{\rho \leq r}$ is strictly outward minimising for $r$ large enough.
\end{lemma}

\proof Consider any $\varphi \in \CC^\infty_c(\set{\rho \geq r})$, then
\begin{align}
    \int\limits_{\set{\rho \geq r}} \div\left( \frac{ \D \rho }{ \abs{ \D \rho}} \right)\varphi \dd {\mu} = - \int\limits_{\set{\rho \geq r}} \ip{\frac{ \D \rho }{ \abs{ \D \rho}}, \D \varphi}\dd {\mu}  - \int\limits_{\set{\rho =r }}\varphi \dd {\sigma_g}.
\end{align}
Observe that the right hand side of the previous identity depends only on the coefficient of the metric and not on their derivatives. Since the metric $g$ converges to the metric $\hat{g}$, for every $\varepsilon>0$ there exists $R_\varepsilon$ such that for every $r \geq R_\varepsilon$
\begin{equation}\label{eq:convergence_mean_curvature}
    \abs{\; \int\limits_{M} \div\left( \frac{ \D \rho }{ \abs{ \D \rho}}\right) \varphi \dd {\mu_g} - \int\limits_{M} \frac{n-1}{\rho} \varphi \dd {\mu_{\hat{g}}}} \leq \varepsilon \int\limits_{M} \frac{n-1}{\rho} \varphi \dd {\mu_{\hat{g}}}
\end{equation}
for every $\varphi \in \CC^\infty_c(\set{\rho \geq r})$ and
\begin{equation}
    \abs{ \abs{E}_g - \abs{E}_{\hat{g}}} \leq \varepsilon \abs{ E}_{\hat{g}}
\end{equation}
for every measurable $E\subset \set{\rho \geq r}$. By \eqref{eq:convergence_mean_curvature} and the density of compactly supported smooth function, for every $E\subset \set{\rho \geq r}$ we have that
\begin{equation}
     \int\limits_{E} \div\left(\frac{\D \rho}{\abs{\D \rho}}\right) \dd {\mu_g}\geq (1-\varepsilon)\frac{n-1}{\sup_{E} \rho} \abs{E}_{\hat{g}}.
\end{equation}
Let $F$ be a subset of finite perimeter containing $\set{\rho < r}$, then
 \begin{align}
     \left(\frac{1-\varepsilon}{1+ \varepsilon}\right)\left(\frac{n-1}{\sup_{F} \rho}\right)\abs{F\smallsetminus\set{\rho < r}}_{g}& \leq (1-\varepsilon) \frac{n-1}{\sup_{F} \rho}\abs{F\smallsetminus\set{\rho \leq r}}_{\hat{g}} \leq \int\limits_{M} \div \left( \frac{\D \rho}{\abs{ \D \rho }}\right) \dd {\mu_g}\\
     &\leq \int\limits_{\partial^* F} \ip{  \frac{\D \rho}{\abs{ \D \rho }}, \nu_{\partial^* F} } \dd{\sigma_g} - \int\limits_{\set{\rho=r}} \ip{ \frac{\D \rho}{\abs{ \D \rho }}, \nu_{\set{\rho = r}}} \dd{\sigma_g}\\
     &\leq \abs{ \partial^* F}_g - \abs{ \set{\rho =r}}_g.
 \end{align}
 This proves that $\abs{ \set{\rho =r}}_g\leq  \abs{ \partial^* F}_g $ and the equality holds true if and only if $\abs{F\smallsetminus\set{\rho < r}}_{g} = 0$, which gives that $\set{\rho \leq r}$ is strictly outward minimising. \endproof 

By \cite[Theorem 2.16]{Fogagnolo2020a} the existence of an exhaustion of strictly outward minimising sets provides, for any bounded $\Omega \subset M$ with smooth boundary, a suitable bounded \emph{strictly outward minimising hull} $\Omega^*$, that in particular fulfils
\begin{equation}
\abs{\partial \Omega^*} = \inf\set{\abs{\partial^*F} \st F \text{ is bounded and } \Omega \subseteq F}.
\end{equation}
Let $(M,g)$ be a $\CC^0$-Asymptotically Conical Riemannian manifold defined in \cref{def:C0ACintro} and denote by $E_1,\ldots, E_N$ the (finitely many) ends. Consider $\Omega\subset M$ open bounded subset with smooth boundary and $w:M \smallsetminus \Omega \to [0,+\infty)$ the weak IMCF $w$ starting at $\Omega$. As we did for the $p$-capacity we can define the area of the strictly outward minimising hull of $\Omega$ with respect to one end $E_i$. Indeed, there exists a time $T$ such that $\set{w \leq t}$ contains the compact $K$ defined in \cref{def:C0ACintro} for every $t \geq T$. We then define the area of $\partial \Omega^* $ with respect to $E_i$ as 
\begin{equation}
    \abs{\partial \Omega^*}^{(i)}= \frac{\abs{\partial\set{w\leq t}\cap E_i}}{\ee^t}
\end{equation}
for some $t \geq T$. Observe that such definition is well posed by \cite[Exponential Growth Lemma 1.6]{Huisken2001}. Moreover, it can be checked, substantially using \cite[Proposition 3.4]{Fogagnolo2020a}, that  $\abs{ \partial \Omega^*}$ splits as
\begin{equation}
    \abs{\partial \Omega^*} = \sum_{i=1}^N \abs{\partial \Omega^*}^{(i)}.
\end{equation}
Theorem 1.2 in \cite{Fogagnolo2020a} establishes a relation between $\Cn_p(\Omega)$ and $\abs{\partial \Omega^*}$. The same relation between $\Cn^{(i)}_p(\Omega)$ and $\abs{\partial \Omega^*}^{(i)}$ holds true.
\begin{lemma}\label{lem:convergenceipcap}
    Let $(M,g)$ be a $\CC^0$-Asymptotically Conical Riemannian manifold with Ricci curvature satisfying \eqref{eq:Ricci-intro} and $N$ be the number of ends. Let $\Omega \subseteq M$ be an open bounded subset with smooth boundary. Then,
    \begin{equation}
        \lim_{p\to 1^+} \Cn^{(i)}_p(\Omega)= \frac{ \abs{ \partial \Omega^*}^{(i)}}{ \abs{\S^{n-1}}}
    \end{equation}
holds for every $i=1, \ldots, N$
\end{lemma}

\proof Let $w$ be the solution to the weak IMCF starting at $\Omega$ and $T$ large enough so that $\set{w \leq T}$ contains $K$ in \cref{def:C0ACintro}. Given $u_p$ the solution to \eqref{pb-intro}, by \cite{Mari2019} we have $-(p-1)\log u_p$ where converges locally uniformly to $w$ as $p\to 1^+$. In particular, for every $t\geq T$ there exists $p_t\in (1,n)$ such that $\set{w\leq T}\subseteq \set{-(1-p)\log u_p\leq t}$ holds for every $p<p_t$. Arguing as in \cite[Theorem 1.2]{Fogagnolo2020a}, since an Isoperimetric Inequality is in force by \cref{prop:Sobolev_in_AC}, we can prove that $\abs{\partial \set{w\leq t}\cap E_i} \leq \Cn_{n,p}\,\capa_p(\partial \set{w\leq t};E_i)$, for some constant $\Cn_{n,p}$ such that $\Cn_{n,p}\to 1$ as $p\to 1^+$. In particular, by the monotonicity of the $p$-capacity and \cref{prop:scaling-invariant-cap} we have that
\begin{align}
    \abs{\partial \Omega^*}^{(i)} &\leq \Cn_{n,p}\,\ee^{-T}\capa_p\left( \set{w\leq T}\cap E_i; E_i\right) \leq \ee^{-T}\capa_p\left(\set{u_p\geq \Cn_{n,p}\,\ee^{-\frac{t}{p-1}}}\cap E_i;E_i\right)\\& =\Cn_{n,p}\, \ee^{t-2T} \capa_p\left(\set{u_p\geq \ee^{-\frac{T}{p-1}}}\cap E_i ;E_i\right) \leq \Cn_{n,p}\ee^{t-T} \capa_p^{(i)}( \Omega)
\end{align}
Sending $p\to 1^+$ and then $t\to T^+$ we have that
\begin{equation}
    \frac{\abs{\partial \Omega^*}^{(i)}}{\abs{\S^{n-1}}} \leq \lim_{p\to 1^+}\Cn_p^{(i)}( \Omega).
\end{equation}
If for some $i=1, \ldots, N$ the inequality is strict, then
\begin{equation}
    \frac{\abs{\partial \Omega^*}}{\abs{\S^{n-1}}}=\sum_{i=1}^N    \frac{\abs{\partial \Omega^*}^{(i)}}{\abs{\S^{n-1}}} <  \sum_{i=1}^N\lim_{p\to 1^+}\Cn_p^{(i)}( \Omega) =\lim_{p\to 1^+}\Cn_p( \Omega) = \frac{\abs{\partial \Omega^*}}{\abs{\S^{n-1}}}
\end{equation}
which is a contradiction. \endproof

Clearly, we also obtain the analogous of \cref{prop:p-capacity_level_sets}.

\begin{proposition}[Asymptotic behaviour of the area of level sets]\label{prop:area_level_set_IMCF}
In the same assumptions and notations of \cref{thm:asymptotic_behaviour_of_IMCF}, set, for $i=1,\ldots,m$,
\begin{equation}
    v_i= \left(\frac{\abs{\partial \Omega^*}^{(i)}}{ \abs{ \S^{n-1}}\AVR(g; E_i)}\right)^{\frac{1}{n-1}} \ee^{\frac{w}{n-1}}.
\end{equation}
Then, we have
\begin{equation}
    \lim_{s \to +\infty}\frac{\abs{\set{v_i=s}\cap E_i}}{s^{n-1} \abs{ \S^{n-1}}} = \AVR(g).
\end{equation}
\end{proposition}
Observe that a similar result for the $p$-capacitary potential was obtained in \cref{eq:area_level_sets}. In that case a first order asymptotic behaviour for the $p$-capacitary potential was required in the proof. The reason is that the area is linked to the level sets of IMCF in the same way the $p$-capacity is linked to the level set of $p$-capacitary potential. A simple $\CC^0$-convergence is therefore enough in this case.

As a byproduct we also obtain the counterpart of \cref{prop:uniqueness_on_cones} proving that
\begin{align}\label{eq:unique_solution_cone}
w(x)= (n-1) \log(\rho(x))&&\text{with }x\in(0,+\infty)\times L
\end{align}
is the unique solution on $(0,+\infty)\times L$ up to a constant.

\begin{proposition}\label{prop:uniqueness-on-cones-IMCF}
Let $((0,+\infty)\times L, \hat{g})$ be a Riemannian cone with nonnegative Ricci curvature where $L$ is a closed connected smooth hypersurface. Let $w$ be a weak IMCF on $(0,+\infty)\times L$ satisfying $w(x) \geq  (n-1) \log \rho(x) +\Cn$ for every $x \in (0,+\infty) \times L$ for some constant $\Cn\geq 0$. Then, there exists a $\gamma \in \R$ such that
\begin{align}
    w(x)= (n-1)\log(\rho(x))+\gamma && \text{with } x \in (0,+\infty)\times L,
\end{align}
holds on $(0,+\infty) \times L$.
\end{proposition}
 In the case of the flat Euclidean space, this uniqueness result has been obtained in \cite[Proposition 7.2]{Huisken2001}. 
\proof[Proof of \cref{thm:asymptotic_behaviour_of_IMCF,prop:area_level_set_IMCF,prop:uniqueness-on-cones-IMCF}] The proof follows the same lines of \cref{thm:asymptoticbehaviourofppotentialC0}. We prove the theorem in the case $M$ has only one end, since the general case follows applying the result to each end. We denote by $g_{(s)}$ the metric $s^{-2} \omega_s^*g$ on $[1/s, +\infty) \times L$, being $\omega_s$ the family of diffeomorphism defined in \cref{eq:dilatation}. We divide the proof in three steps. The second and the third one contains the proof of \cref{prop:area_level_set_IMCF,,prop:uniqueness-on-cones-IMCF} respectively. 

\step Define for every $s \geq 1 $ the family of functions $w_s: [1/s,+\infty) \times L \to \R$ as
\begin{equation}
    w_s= w\circ \omega_s - (n-1) \log (s)
\end{equation}
where $\omega_s$ is the map defined in \cref{eq:dilatation}. Employing \eqref{eq:li-yau-imcf-C0AC} in \cref{thm:weak_imcf_AC_to_Riccinonneg} and \cref{prop:Cheng-Yau-Improved} as in the proof of \cref{thm:asymptoticbehaviourofppotentialC0}, it is easy to show that $(w_s)_{s \geq 1}$ is equibounded and equi-Lipschitz. By Arzelà-Ascoli Theorem, $(w_s)_{s \geq 1}$ is precompact with respect to the local uniform convergence on $(0,+\infty) \times L$. Moreover, by \cite[Compactness Theorem 2.1]{Huisken2001} every limit point $u$ is a solution to the (weak) IMCF on $(0,+\infty)\times L$ and by \cref{eq:li-yau-imcf-C0AC} there exists a positive constant $\Cn>0$ such that
\begin{equation}
(n-1) \log (\rho (x)) -\Cn \leq u(x) \leq (n-1) \log(\rho(x)) + \Cn
\end{equation}
is satisfied on $(0,+\infty)\times L$.

\step Here we prove \cref{prop:uniqueness-on-cones-IMCF}, inferring in particular that any limit point $v$ of $(w_s)_{s \geq 1}$ satisfies
\begin{equation}
    v(x)= (n-1) \log \rho (x) + \gamma
\end{equation}
on $(0,+\infty) \times L$ for some $\gamma \in \R$. Let $\ecc_v:\R\to \R $ be defined as
\begin{equation}
    \ecc_v(t)= \frac{ R(t)}{r(t)}
\end{equation}
where, for every $t \in\R$, $[r(t), R(t)] \times L$ is the smallest annulus containing $\set{v=t}$. Arguing as in \cref{step:gammadefconvergence} of \cref{thm:asymptoticbehaviourofppotentialC0} starting from any weak IMCF $v$ on $(0,+\infty) \times L$ we can produce a function $u:(0,+\infty) \times L \to \R$ such that $\ecc_u(t)$ is constant and is equal to $\ecc_u= \sup_t \ecc_v(t)\in [1,+\infty)$. Suppose by contradiction that $\ecc_u>1$. Then the level $\set{u=0}\subset [r(0), \ecc_u r(0)]\times L$ and touches both the cross-sections $\set{\rho =r(0)}$ and $\set{\rho = \ecc_u r(0)}$ without being equal to either one. We aim to compare the weak flow and the two strong flows and prove that the level sets of $u$ detach from the spheres. Perturb $\set{\rho \leq r(0)}$ outward and $\set{\rho \leq \ecc_u r(0)}$ inward to obtain $D^-$ and $D^+$ respectively with the following properties:
\begin{itemize}
    \item $\set{\rho \leq r(0)}\subset D^-\subset \set{u\leq 0 }$ and $\set{u\leq 0 }\subset D^+\subset \set{\rho \leq \ecc_u r(0)}$;
    \item $D^-$ and $D^+$ are starshaped with smooth strictly mean convex boundary.
\end{itemize}
Then the smooth IMCF starting at $D^+$ and $D^-$ exists for all time by \cite[Theorem 3.1]{Zhou2018} and by \cite[Smooth Flow Lemma 2.3] {Huisken2001} it coincides with the weak notion of the IMCF. Denote by $(D^-_t)_{t \geq 0}$ and  $(D^+_t)_{t \geq 0}$ the sublevel sets of the two weak (and smooth) IMCF starting at $D^-$ and $D^+$ respectively. By the Strong Comparison Principle for smooth flows we have that
\begin{align}\label{eq:firstcomparison}
    \rho(x) > r(0) \ee^{\frac{t}{n-1}} \text{ for }x\in\partial D^-_t && \text{and} && \rho(x) < \ecc_u r(0) \ee^{\frac{t}{n-1}} \text{ for }x \in\partial D^+_t.
\end{align}
On the other hand, by the Weak Comparison theorem \cite[Theorem 2.2(ii)]{Huisken2001}
\begin{equation}\label{eq:secondcomparison}
    D^-_t \subset \set{u\leq t} \subset D^+_t.
\end{equation}
Coupling \eqref{eq:firstcomparison} and \eqref{eq:secondcomparison} we have that $\ecc_u(t)< \ecc_u$ which is the desired contradiction. Then $\ecc_u=1$ that completes, as in \cref{step:gammadefconvergence} of \cref{thm:asymptoticbehaviourofppotentialC0}, the proof of \cref{prop:uniqueness-on-cones-IMCF}.

\step  Let $v= (n-1) \log \rho + \gamma$ be a limit point of the family $(w_s)_{s \geq1}$. We are now going to prove that
\begin{equation}\label{eq:characterisation_gamma_IMCF}
    \gamma = \log\left( \frac{\AVR(g) \abs{\S^{n-1}}}{\abs{\partial \Omega^*}} \right).
\end{equation}
The characterisation proves \cref{prop:area_level_set_IMCF} and implies that the limit point is unique, concluding the proof. We work with the auxiliary function
\begin{equation}
    u= \ee^{\frac{w-\gamma}{n-1}}
\end{equation}
Since $v$ is a limit point for the family $(w_s)_{s\geq1}$ there exists a subsequence $(w_{s_k})_{k \in \N}$, $s_k$ increasing and divergent as $k\to +\infty$, such that $w_{s_k}\to v= (n-1) \log \rho + \gamma$ locally uniformly on $(0,+\infty) \times L$ as $k\to +\infty$. Then for any $\varepsilon>0$ there exists $k_\varepsilon\in \N$ such that
\begin{equation}
   \set{ \rho \leq \frac{s_k}{1+ \varepsilon}} \subset \set{ u \leq s_k} \subset \set{ \rho \leq \frac{s_k}{1- \varepsilon}}
\end{equation}
holds for every $k \geq k_\varepsilon$. By \cref{lem:cross-sections-SOM} we can assume $k_\varepsilon$ large enough so that both the left most and the right most sets are strictly outward minimising for any $k\geq k_\varepsilon$. Then the perimeter is monotone by inclusion and by \cite[Exponential Growth Lemma 1.6]{Huisken2001} we have
\begin{equation}
    \abs{\set{\rho = \frac{s_k}{1+ \varepsilon}}} \leq \ee^\gamma s_k^{n-1} \abs{ \partial \Omega^*}\leq \abs{\set{\rho = \frac{s_k}{1- \varepsilon}} }
\end{equation}
Dividing both sides by $\abs{\S^{n-1}}s_k^{n-1}$, sending $k \to +\infty$ and using \cref{lem:AVRconico} we infer that
\begin{equation}
    \frac{\AVR(g)}{(1+\varepsilon)^{n-1}} \leq \ee^{\gamma} \frac{\abs{\partial \Omega^*}}{\abs{\S^{n-1}}}\leq \frac{\AVR(g)}{(1-\varepsilon)^{n-1}}  
\end{equation}
Then \eqref{eq:characterisation_gamma_IMCF} follows by arbitrariness of $\varepsilon>0$. \endproof

Firstly, observe that we do not have the analogous of \cref{thm:asymptoticbehaviourofpotential} for the IMCF. The asymptotic behaviour of higher order derivatives of the $p$-capacitary potential is indeed a consequence of the higher regularity one gets once shown that the gradient does not vanish anymore sufficiently far out. In the case of the weak  IMCF, it is not clear to us how to work out this last step.

The result above is to be compared with \cite[Lemma 7.1]{Huisken2001}. We obtain here an explicit characterisation of the constants $c_\lambda$ that is
\begin{equation} 
    c_\lambda = -(n-1)\log\left(\frac{\abs{\S^{n-1}}}{\abs{\partial \Omega^*}}\lambda\right).
\end{equation}
The constant appearing in \cref{eq:asymptotic_behaviour_IMCF} satisfies
\begin{equation}
   \log\left( \frac{\AVR(g) \abs{\S^{n-1}}}{\abs{\partial \Omega^*}}\right) =\lim_{p \to 1^+} -(p-1) \log \left[\left( \frac{\capa_p(\partial \Omega)}{\AVR(g)}\right)^{\frac{1}{p-1}}\right]
\end{equation}
thanks to \cref{eq:convergence_p-cap_potential_SOMH}, which is the constant in \cref{eq:p-cap_asymptotic_behaviour_C0}, trasformed accordingly to $w_p=-(p-1)\log u_p$. Hence, even if by \cref{thm:weak_imcf_AC_to_Riccinonneg} $w_p\to w$ only locally uniformly as $p\to 1^+$, the asymptotic behaviour of $w$ is anyway affected by this procedure.

\printbibliography

@article{Federer1960,
  title={Normal and integral currents},
  author={Federer, Herbert and Fleming, Wendell H},
  journal={Annals of Mathematics},
  pages={458--520},
  year={1960},
  publisher={JSTOR}
}

@article{Varopoulos1985,
title = {Hardy-Littlewood theory for semigroups},
journal = {Journal of Functional Analysis},
volume = {63},
number = {2},
pages = {240-260},
year = {1985},
issn = {0022-1236},
doi = {https://doi.org/10.1016/0022-1236(85)90087-4},
url = {https://www.sciencedirect.com/science/article/pii/0022123685900874},
author = {Varopoulos, Nicholas Theodore}
}

@PHDTHESIS{Carron1994,
url = "http://www.theses.fr/1994GRE10107",
title = "Inégalités isopérimétriques sur les variétés riemanniennes",
author = "Carron, Gilles",
year = "1994",
pages = "1 vol. (77 P.)",
note = "Thèse de doctorat dirigée par Gallot, Sylvestre Mathématiques Grenoble 1 1994",
}

@book {Hebey1999,
    AUTHOR = {Hebey, Emmanuel},
     TITLE = {Nonlinear analysis on manifolds: {S}obolev spaces and
              inequalities},
    SERIES = {Courant Lecture Notes in Mathematics},
    VOLUME = {5},
 PUBLISHER = {New York University, Courant Institute of Mathematical
              Sciences, New York; American Mathematical Society, Providence,
              RI},
      YEAR = {1999},
     PAGES = {x+309},
      ISBN = {0-9658703-4-0},
   MRCLASS = {58D15 (35J60 46E35 53C21 58J60)},
  MRNUMBER = {1688256},
MRREVIEWER = {Gilles Carron},
}

@Article{Heinonen1988,
  author    = {Heinonen, Juha and Kilpel{\"a}inen, Tero},
  journal   = {Arkiv f{\"o}r Matematik},
  title     = {A-superharmonic functions and supersolutions of degenerate elliptic equations},
  year      = {1988},
  number    = {1},
  pages     = {87--105},
  volume    = {26},
  publisher = {Springer Netherlands},
}

@book {Greene1979,
    AUTHOR = {Greene, R. E. and Wu, H.},
     TITLE = {Function theory on manifolds which possess a pole},
    SERIES = {Lecture Notes in Mathematics},
    VOLUME = {699},
 PUBLISHER = {Springer, Berlin},
      YEAR = {1979},
     PAGES = {ii+215},
      ISBN = {3-540-09108-4},
   MRCLASS = {53-02 (32F99 32H20 53C20 53C55 58E20)},
  MRNUMBER = {521983},
MRREVIEWER = {M. L. Gromov},
}

@article {DePhilippis2018,
    AUTHOR = {De Philippis, Guido and Gigli, Nicola},
     TITLE = {Non-collapsed spaces with {R}icci curvature bounded from
              below},
   JOURNAL = {J. \'{E}c. polytech. Math.},
  FJOURNAL = {Journal de l'\'{E}cole polytechnique. Math\'{e}matiques},
    VOLUME = {5},
      YEAR = {2018},
     PAGES = {613--650},
      ISSN = {2429-7100},
   MRCLASS = {53C23 (53C21)},
  MRNUMBER = {3852263},
MRREVIEWER = {Shouhei Honda},
       DOI = {10.5802/jep.80},
       URL = {https://doi.org/10.5802/jep.80},
}

@article {Gigli2015b,
    AUTHOR = {Gigli, Nicola and Mondino, Andrea and Savar\'{e}, Giuseppe},
     TITLE = {Convergence of pointed non-compact metric measure spaces and
              stability of {R}icci curvature bounds and heat flows},
   JOURNAL = {Proc. Lond. Math. Soc. (3)},
  FJOURNAL = {Proceedings of the London Mathematical Society. Third Series},
    VOLUME = {111},
      YEAR = {2015},
    NUMBER = {5},
     PAGES = {1071--1129},
      ISSN = {0024-6115},
   MRCLASS = {53C23 (28A33 53C21 58J35)},
  MRNUMBER = {3477230},
       DOI = {10.1112/plms/pdv047},
       URL = {https://doi.org/10.1112/plms/pdv047},
}

@Article{Huisken2001,
  author     = {Huisken, Gerhard and Ilmanen, Tom},
  title      = {The inverse mean curvature flow and the {R}iemannian {P}enrose inequality},
  journal    = {J. Differential Geom.},
  year       = {2001},
  volume     = {59},
  number     = {3},
  pages      = {353--437},
  issn       = {0022-040X},
  fjournal   = {Journal of Differential Geometry},
  mrclass    = {53C44 (35D05 35J20 35J60 53C21 53C42 83C57)},
  mrnumber   = {1916951},
  mrreviewer = {John Urbas},
  url        = {http://projecteuclid.org/euclid.jdg/1090349447},
}

@Article{Colesanti2015,
  author    = {Colesanti, Andrea and Nystr{\"o}m, Kaj and Salani, Paolo and Xiao, Jie and Yang, Deane and Zhang, Gaoyong},
  title     = {The {H}adamard variational formula and the {M}inkowski problem for p-capacity},
  journal   = {Advances in Mathematics},
  year      = {2015},
  volume    = {285},
  pages     = {1511--1588},
  publisher = {Elsevier},
}

@Article{Lieberman1988,
  author    = {Lieberman, Gary M},
  title     = {Boundary regularity for solutions of degenerate elliptic equations},
  journal   = {Nonlinear Analysis: Theory, Methods \& Applications},
  year      = {1988},
  volume    = {12},
  number    = {11},
  pages     = {1203--1219},
  publisher = {Elsevier},
}

@Article{Moser2008,
  author    = {Moser, Roger},
  title     = {The inverse mean curvature flow as an obstacle problem},
  journal   = {Indiana University Mathematics Journal},
  year      = {2008},
  pages     = {2235--2256},
  publisher = {JSTOR},
}

@Article{Tolksdorf1983,
  author    = {Peter Tolksdorf},
  journal   = {Communications in Partial Differential Equations},
  title     = {On The {D}irichlet problem for Quasilinear Equations},
  year      = {1983},
  number    = {7},
  pages     = {773-817},
  volume    = {8},
  doi       = {10.1080/03605308308820285},
  publisher = {Taylor \& Francis},
}

@Article{Holopainen1999,
  author    = {Holopainen, Ilkka},
  title     = {Volume growth, {G}reen’s functions, and parabolicity of ends},
  journal   = {Duke mathematical journal},
  year      = {1999},
  volume    = {97},
  number    = {2},
  pages     = {319--346},
  publisher = {Duke University Press},
}

@Article{Colding1997,
  author    = {Tobias H. Colding and William P. Minicozzi},
  title     = {Large Scale Behavior of Kernels of Schrödinger Operators},
  journal   = {American Journal of Mathematics},
  year      = {1997},
  volume    = {119},
  number    = {6},
  pages     = {1355--1398},
  issn      = {00029327, 10806377},
  publisher = {Johns Hopkins University Press},
  url       = {http://www.jstor.org/stable/25098578},
}

@InProceedings{Fogagnolo2019,
  author       = {Fogagnolo, Mattia and Mazzieri, Lorenzo and Pinamonti, Andrea},
  title        = {Geometric aspects of p-capacitary potentials},
  booktitle    = {Annales de l'Institut Henri Poincar{\'e} C, Analyse non lin{\'e}aire},
  year         = {2019},
  volume       = {36},
  number       = {4},
  pages        = {1151--1179},
  organization = {Elsevier},
}

@Article{Li1997,
  author    = {Li, Peter and Tam, Luen-Fai and Wang, Jiaping},
  title     = {Sharp bounds for the Green’s function and the heat kernel},
  journal   = {Mathematical Research Letters},
  year      = {1997},
  volume    = {4},
  number    = {4},
  pages     = {589--602},
  publisher = {International Press of Boston},
}

@Book{Gilbarg2015,
  title     = {Elliptic partial differential equations of second order},
  publisher = {springer},
  year      = {2015},
  author    = {Gilbarg, David and Trudinger, Neil S},
}

@article {Li1992,
    AUTHOR = {Li, Peter and Tam, Luen-Fai},
     TITLE = {Harmonic functions and the structure of complete manifolds},
   JOURNAL = {J. Differential Geom.},
  FJOURNAL = {Journal of Differential Geometry},
    VOLUME = {35},
      YEAR = {1992},
    NUMBER = {2},
     PAGES = {359--383},
      ISSN = {0022-040X},
   MRCLASS = {53C21 (53C20 58G30)},
  MRNUMBER = {1158340},
MRREVIEWER = {Yang Lian Pan},
       URL = {http://projecteuclid.org/euclid.jdg/1214448079},
}

@Article{Wang2010,
  author  = {Xiaodong Wang and Lei Zhang},
  journal = {Communications in Analysis and Geometry},
  title   = {Local gradient estimate for $p$-harmonic functions on Riemannian manifolds},
  year    = {2010},
  pages   = {759-771},
  volume  = {19},
}

@Misc{Huisken2008,
  author    = {Huisken , G. and Ilmanen, T.},
  month     = {11},
  title     = {Higher regularity of the inverse mean curvature flow},
  year      = {2008},
  doi       = {10.4310/jdg/1226090483},
  fjournal  = {Journal of Differential Geometry},
  journal   = {J. Differential Geom.},
  number    = {3},
  pages     = {433--451},
  publisher = {Lehigh University},
  url       = {https://doi.org/10.4310/jdg/1226090483},
  volume    = {80},
}

@Misc{Mari2019,
  archiveprefix = {arXiv},
  author        = {Mari, Luciano and Rigoli, Marco and Setti, Alberto Giulio},
  title         = {On the 1/H-flow by p-Laplace approximation: new estimates via fake distances under Ricci lower bounds},
  year          = {2021},
  eprint        = {1905.00216},
  primaryclass  = {math.DG},
  note={To appear in \emph{American Journal of Mathematics}}
}

@Article{Kichenassamy1986,
  author    = {Kichenassamy, Satyanad and V{\'e}ron, Laurent},
  title     = {Singular solutions of the p-{L}aplace equation},
  journal   = {Mathematische Annalen},
  year      = {1986},
  volume    = {275},
  number    = {4},
  pages     = {599--615},
  publisher = {Springer},
}

@article {Ding2002,
    AUTHOR = {Ding, Yu},
     TITLE = {Heat kernels and {G}reen's functions on limit spaces},
   JOURNAL = {Comm. Anal. Geom.},
  FJOURNAL = {Communications in Analysis and Geometry},
    VOLUME = {10},
      YEAR = {2002},
    NUMBER = {3},
     PAGES = {475--514},
      ISSN = {1019-8385},
   MRCLASS = {58J35 (35A08 53C23)},
  MRNUMBER = {1912256},
MRREVIEWER = {Man Chun Leung},
       DOI = {10.4310/CAG.2002.v10.n3.a3},
       URL = {https://doi.org/10.4310/CAG.2002.v10.n3.a3},
}

@Article{DiBenedetto1983,
  author      = {DiBenedetto, Emmanuele},
  title       = {${C}^{1+\alpha}$ {L}ocal {R}egularity of {W}eak {S}olutions of {D}egenerate {E}lliptic {E}quations.},
  year        = {1983},
  issn        = {0362-546X},
  number      = {8},
  pages       = {827--850},
  volume      = {7},
  doi         = {https://doi.org/10.1016/0362-546X(83)90061-5},
  institution = {Nonlinear Analysis: Theory, Methods & Applications},
  url         = {https://www.sciencedirect.com/science/article/pii/0362546X83900615},
}

@Article{Moser2007,
  author  = {Moser, Roger},
  journal = {Journal of the European Mathematical Society},
  title   = {The inverse mean curvature flow and $p$-harmonic functions},
  year    = {2007},
  number  = {1},
  pages   = {77--83},
  volume  = {9},
}

@Article{Hawking1977,
  author   = {Hawking, S. W.},
  journal  = {Phys. Lett. A},
  title    = {Gravitational instantons},
  year     = {1977},
  issn     = {0375-9601},
  number   = {2},
  pages    = {81--83},
  volume   = {60},
  doi      = {10.1016/0375-9601(77)90386-3},
  fjournal = {Physics Letters. A},
  mrclass  = {83.53},
  mrnumber = {0465052},
}

@Article{Eguchi1979,
  author     = {Eguchi, T. and Hanson, A. J.},
  journal    = {Ann. Physics},
  title      = {Self-dual solutions to {E}uclidean gravity},
  year       = {1979},
  issn       = {0003-4916},
  number     = {1},
  pages      = {82--106},
  volume     = {120},
  doi        = {10.1016/0003-4916(79)90282-3},
  fjournal   = {Annals of Physics},
  mrclass    = {83C20 (57R25)},
  mrnumber   = {540896},
  mrreviewer = {A. F. da F. Teixeira},
}

@Article{Kronheimer1989,
  author     = {Kronheimer, P. B.},
  journal    = {J. Differential Geom.},
  title      = {A {T}orelli-type theorem for gravitational instantons},
  year       = {1989},
  issn       = {0022-040X},
  number     = {3},
  pages      = {685--697},
  volume     = {29},
  fjournal   = {Journal of Differential Geometry},
  mrclass    = {53C25 (14B07 32L10 32M10 53C55 83C20)},
  mrnumber   = {992335},
  mrreviewer = {Krzysztof Galicki},
  url        = {http://projecteuclid.org/euclid.jdg/1214443067},
}

@Article{Kronheimer1989a,
  author     = {Kronheimer, P. B.},
  journal    = {J. Differential Geom.},
  title      = {The construction of {ALE} spaces as hyper-{K}\"ahler quotients},
  year       = {1989},
  issn       = {0022-040X},
  number     = {3},
  pages      = {665--683},
  volume     = {29},
  fjournal   = {Journal of Differential Geometry},
  mrclass    = {53C25 (14B07 32L10 32M10 53C55 83C20)},
  mrnumber   = {992334},
  mrreviewer = {Krzysztof Galicki},
  url        = {http://projecteuclid.org/euclid.jdg/1214443066},
}

@Article{Minerbe2009,
  author     = {Minerbe, V.},
  journal    = {Comm. Math. Phys.},
  title      = {A mass for {ALF} manifolds},
  year       = {2009},
  issn       = {0010-3616},
  number     = {3},
  pages      = {925--955},
  volume     = {289},
  doi        = {10.1007/s00220-009-0823-3},
  fjournal   = {Communications in Mathematical Physics},
  mrclass    = {53C21 (53C27 53C80 58J60)},
  mrnumber   = {2511656},
  mrreviewer = {Gabjin Yun},
}

@Article{Minerbe2010,
  author     = {Minerbe, V.},
  journal    = {Ann. Sci. \'Ec. Norm. Sup\'er. (4)},
  title      = {On the asymptotic geometry of gravitational instantons},
  year       = {2010},
  issn       = {0012-9593},
  number     = {6},
  pages      = {883--924},
  volume     = {43},
  doi        = {10.24033/asens.2135},
  fjournal   = {Annales Scientifiques de l'\'Ecole Normale Sup\'erieure. Quatri\`eme S\'erie},
  mrclass    = {53C26 (53C80)},
  mrnumber   = {2778451},
  mrreviewer = {Derek G. Harland},
}

@Article{Minerbe2011,
  author     = {Minerbe, V.},
  journal    = {J. Reine Angew. Math.},
  title      = {Rigidity for multi-{T}aub-{NUT} metrics},
  year       = {2011},
  issn       = {0075-4102},
  pages      = {47--58},
  volume     = {656},
  doi        = {10.1515/CRELLE.2011.042},
  fjournal   = {Journal f\"ur die Reine und Angewandte Mathematik. [Crelle's Journal]},
  mrclass    = {53C26 (53C24 53C80)},
  mrnumber   = {2818855},
  mrreviewer = {Hayde\'e Herrera},
}

@Book{Ladyzhenskaia1968,
  author    = {Ladyzhenskaia, Olga Aleksandrovna and Solonnikov, Vsevolod Alekseevich and Ural'tseva, Nina N},
  publisher = {American Mathematical Soc.},
  title     = {Linear and quasi-linear equations of parabolic type},
  year      = {1968},
  volume    = {23},
}

@Article{Urbas1990,
  author   = {Urbas, J. I.E.},
  journal  = {Mathematische Zeitschrift},
  title    = {On the expansion of starshaped hypersurfaces by symmetric functions of their principal curvatures.},
  year     = {1990},
  number   = {3},
  pages    = {355-372},
  volume   = {205},
  url      = {http://eudml.org/doc/174181},
}

@Article{Holopainen1990,
  author    = {Holopainen, I.},
  journal   = {Annales Academiae Scientiarum Fennicae. Series A 1. Mathematica. Dissertationes},
  title     = {Nonlinear potential theory and quasiregular mappings on Riemannian manifolds},
  year      = {1990},
  issn      = {0355-0087},
  pages     = {1--45},
  volume    = {74},
  language  = {English},
  publisher = {Suomalainen Tiedeakatemia},
}

@Article{Chodosh2017,
  author    = {Chodosh, Otis and Eichmair, Michael and Volkmann, Alexander},
  journal   = {Journal of Differential Geometry},
  title     = {Isoperimetric structure of asymptotically conical manifolds},
  year      = {2017},
  number    = {1},
  pages     = {1--19},
  volume    = {105},
  publisher = {Lehigh University},
}

@Misc{Fogagnolo2020a,
  author        = {Mattia Fogagnolo and Lorenzo Mazzieri},
  title         = {Minimising hulls, p-capacity and isoperimetric inequality on complete Riemannian manifolds},
  year          = {2020},
  archiveprefix = {arXiv},
  eprint        = {2012.09490},
  primaryclass  = {math.DG},
}

@article {Gerhardt1990,
    AUTHOR = {Gerhardt, Claus},
     TITLE = {Flow of nonconvex hypersurfaces into spheres},
   JOURNAL = {J. Differential Geom.},
  FJOURNAL = {Journal of Differential Geometry},
    VOLUME = {32},
      YEAR = {1990},
    NUMBER = {1},
     PAGES = {299--314},
      ISSN = {0022-040X},
   MRCLASS = {53A10 (35K15 58G30)},
  MRNUMBER = {1064876},
MRREVIEWER = {Gerhard Huisken},
       URL = {http://projecteuclid.org/euclid.jdg/1214445048},
}

@Article{Agostiniani2018,
  author   = {Agostiniani, Virginia and Fogagnolo, Mattia and Mazzieri, Lorenzo},
  journal  = {Inventiones mathematicae},
  title    = {Sharp geometric inequalities for closed hypersurfaces in manifolds with nonnegative Ricci curvature},
  year     = {2020},
  issn     = {1432-1297},
  month    = {7},
  day      = {23},
  doi      = {10.1007/s00222-020-00985-4},
}

@inproceedings{Kotschwar2009,
  title={Local gradient estimates of $ p $-harmonic functions, $1/H $-flow, and an entropy formula},
  author={Kotschwar, Brett and Ni, Lei},
  booktitle={Annales scientifiques de l'Ecole normale sup{\'e}rieure},
  volume={42},
  number={1},
  pages={1--36},
  year={2009}
}

@Misc{Benatti2021,
  author        = {Luca Benatti and Mattia Fogagnolo and Lorenzo Mazzieri},
  title         = {{M}inkowski {I}nequality on {A}symptotically {C}onical manifolds},
  year          = {2021},
  archiveprefix = {arXiv},
  eprint        = {2101.06063},
  primaryclass  = {math.DG},
}

@Book{Schoen1994,
  author    = {Schoen, Richard M and Yau, Shing-Tung},
  publisher = {International press Cambridge, MA},
  title     = {Lectures on differential geometry},
  year      = {1994},
  volume    = {2},
}

@article {Chrusciel1990,
    AUTHOR = {Chru\'{s}ciel, Piotr T.},
     TITLE = {Asymptotic estimates in weighted {H}\"{o}lder spaces for a class
              of elliptic scale-covariant second order operators},
   JOURNAL = {Ann. Fac. Sci. Toulouse Math. (5)},
  FJOURNAL = {Toulouse. Facult\'{e} des Sciences. Annales. Math\'{e}matiques. S\'{e}rie
              5},
    VOLUME = {11},
      YEAR = {1990},
    NUMBER = {1},
     PAGES = {21--37},
      ISSN = {0240-2955},
   MRCLASS = {35B45 (35J15 35P20 47F05)},
  MRNUMBER = {1191470},
       URL = {http://www.numdam.org/item?id=AFST_1990_5_11_1_21_0},
}

@misc{Agostiniani2022,
  doi = {10.48550/ARXIV.2205.11642},
  
  url = {https://arxiv.org/abs/2205.11642},
  
  author = {Agostiniani, Virginia and Mantegazza, Carlo and Mazzieri, Lorenzo and Oronzio, Francesca},
  
  title = {Riemannian Penrose inequality via Nonlinear Potential Theory},
  
  publisher = {arXiv},
  
  year = {2022},
  
  copyright = {arXiv.org perpetual, non-exclusive license}
}

@Article{Agostiniani2019,
author={Agostiniani, Virginia
and Fogagnolo, Mattia
and Mazzieri, Lorenzo},
title={Minkowski Inequalities via Nonlinear Potential Theory},
journal={Archive for Rational Mechanics and Analysis},
year={2022},
month={2},
day={11},
issn={1432-0673},
doi={10.1007/s00205-022-01756-6},
url={https://doi.org/10.1007/s00205-022-01756-6}
}

@article {Agostiniani2020,
    AUTHOR = {Agostiniani, Virginia and Mazzieri, Lorenzo and Oronzio,
              Francesca},
     TITLE = {A geometric capacitary inequality for sub-static manifolds
              with harmonic potentials},
   JOURNAL = {Math. Eng.},
  FJOURNAL = {Mathematics in Engineering},
    VOLUME = {4},
      YEAR = {2022},
    NUMBER = {2},
     PAGES = {Paper No. 013, 40},
   MRCLASS = {53C20 (31C12 35C20 53C24 58J05)},
  MRNUMBER = {4281178},
       DOI = {10.3934/mine.2022013},
       URL = {https://doi.org/10.3934/mine.2022013},
}

@misc{Agostiniani2021,
      title={A Green's function proof of the Positive Mass Theorem}, 
      author={V. Agostiniani and L. Mazzieri and F. Oronzio},
      year={2021},
      eprint={2108.08402},
      archivePrefix={arXiv},
      primaryClass={math.DG}
}

@article {Pigola2014,
    AUTHOR = {Pigola, Stefano and Setti, Alberto G. and Troyanov, Marc},
     TITLE = {The connectivity at infinity of a manifold and
              {$L^{q,p}$}-{S}obolev inequalities},
   JOURNAL = {Expo. Math.},
  FJOURNAL = {Expositiones Mathematicae},
    VOLUME = {32},
      YEAR = {2014},
    NUMBER = {4},
     PAGES = {365--383},
      ISSN = {0723-0869},
   MRCLASS = {53C21 (31C12)},
  MRNUMBER = {3279484},
MRREVIEWER = {Qihua Ruan},
       DOI = {10.1016/j.exmath.2013.12.006},
       URL = {https://doi.org/10.1016/j.exmath.2013.12.006},
}

@article {Zhou2018,
    AUTHOR = {Zhou, Hengyu},
     TITLE = {Inverse mean curvature flows in warped product manifolds},
   JOURNAL = {J. Geom. Anal.},
  FJOURNAL = {Journal of Geometric Analysis},
    VOLUME = {28},
      YEAR = {2018},
    NUMBER = {2},
     PAGES = {1749--1772},
      ISSN = {1050-6926},
   MRCLASS = {53C44 (35B45 35J93 35K93 53C42)},
  MRNUMBER = {3790519},
MRREVIEWER = {Julian Scheuer},
       DOI = {10.1007/s12220-017-9887-z},
       URL = {https://doi.org/10.1007/s12220-017-9887-z},
}

@Article{Mantoulidis2020,
  author    = {Mantoulidis, Christos and Miao, Pengzi and Tam, Luen-Fai},
  journal   = {Journal f{\"u}r die reine und angewandte Mathematik (Crelles Journal)},
  title     = {Capacity, quasi-local mass, and singular fill-ins},
  year      = {2020},
  number    = {768},
  pages     = {55--92},
  volume    = {2020},
  publisher = {De Gruyter},
}

@Article{Hirsch2020,
  author    = {Hirsch, Sven and Miao, Pengzi},
  journal   = {Pacific Journal of Mathematics},
  title     = {A positive mass theorem for manifolds with boundary},
  year      = {2020},
  issn      = {0030-8730},
  month     = {6},
  number    = {1},
  pages     = {185–201},
  volume    = {306},
  doi       = {10.2140/pjm.2020.306.185},
  publisher = {Mathematical Sciences Publishers},
}

\end{document}